\documentclass[11pt,reqno]{amsart}
\usepackage{amssymb,amsfonts,amsthm,amscd,stmaryrd,dsfont,esint,upgreek,constants,todonotes}
\usepackage[mathcal]{euscript}
\usepackage[latin1]{inputenc}   
\usepackage[mathcal]{euscript}
\usepackage{graphicx,color}
\numberwithin{equation}{section}

\newcommand{\R}{\mathbb R}
\def\E{\mathbb E}
\def\P{\mathbb P}
\newcommand{\Pw}{\mathcal P_2(\R^d)}

\newconstantfamily{C}{symbol=C}
\newconstantfamily{c}{symbol=c}
\newconstantfamily{O}{symbol=\Omega}

\def\XXint#1#2#3{{\setbox0=\hbox{$#1{#2#3}{\int}$}
\vcenter{\hbox{$#2#3$}}\kern-.5\wd0}}

\numberwithin{equation}{section}
\newtheorem{thm}{Theorem}[section]
\newtheorem{lem}[thm]{Lemma}

\newtheorem{prop}[thm]{Proposition}

\theoremstyle{definition}
\newtheorem{defn}[thm]{Definition}

\def\smallnegint{\mathop{\int\mkern-13mu
        \raise.5ex\hbox{${\scriptscriptstyle\diagup}$}}\nolimits}
\def\ds{\displaystyle}

\def\ep{\varepsilon}

\def\F{{\mathcal F}}

\def\ssetminus{\,\raise.4ex\hbox{$\scriptstyle\setminus$}\,}

\newcommand{\be}{\begin{equation}}
\newcommand{\ee}{\end{equation}}

\newcommand{\bc}{\begin{case}}
\newcommand{\ec}{\end{cases}}
\newcommand{\bs}{\begin{split}}
\newcommand{\es}{\end{split}}

\newcommand{\vs}{\vskip.075in}

\renewcommand{\bar}{\overline}
\renewcommand{\tilde}{\widetilde}
\renewcommand{\hat}{\widehat}

\def\dw{{\bf d}_2}

\begin{document}
\title[Weak solutions for the Master equation with no idiosyncratic noise]{Weak solutions of  the master equation for Mean Field Games with no idiosyncratic noise}

\author[Pierre Cardaliaguet and Panagiotis E. Souganidis]
{Pierre Cardaliaguet and Panagiotis E. Souganidis}
\address{Universit\'e Paris-Dauphine, PSL Research University, Ceremade, 
Place du Mar\'echal de Lattre de Tassigny, 75775 Paris cedex 16 - France}
\email{cardaliaguet@ceremade.dauphine.fr }
\address{Department of Mathematics, University of Chicago, Chicago, Illinois 60637, USA}
\email{souganidis@math.uchicago.edu}
\vskip-0.5in 
\thanks{\hskip-0.149in The second author was partially supported by the National Science Foundation grants DMS-1266383 and DMS-1600129, the Office for Naval Research grant N000141712095 and the Air Force Office for Scientific Research grant FA9550-18-1-0494.}

\dedicatory{Version: \today}

\begin{abstract} We introduce a notion of weak solution  of the  master equation without  idiosyncratic noise in Mean Field Game theory and establish its existence, uniqueness up to a constant and consistency with classical solutions when it is smooth. 
We work  in a monotone setting and rely  on Lions' Hilbert space approach. For the first-order master equation without  idiosyncratic noise, we also give an equivalent definition in the space of measures and establish the well-posedness. 
\end{abstract}

\maketitle      

\section*{Introduction}

We  introduce  a notion of a weak solution of the master equation in the Mean Field Games (MFG for short) theory for first- and second-order models in a monotone setting and without idiosyncratic noise. Using Lions' Hilbert space approach,  we show that the solution exists, is unique up to additive constants, and, when it is smooth, classical. For the first-order master equation without  idiosyncratic noise, we also give an equivalent definition in the space of measures and establish  well-posedness. The arguments do not use any regularity on the solutions which are known only in the presence of idiosyncratic noise. 
\vskip.075in

The master equation in the MFG theory was introduced by Lions in his courses at Coll\`{e}ge de France \cite{LiCoursCollege}. Lions also introduced in \cite{LiCoursCollege} the Hilbertian approach and proved the existence of a classical solution under suitable structure conditions on the coupling function (monotonicity) and Hamiltonian (convexity in the space variable). 
\vs
Defining a notion of well-posed weak solutions for the master equation in MFG is one of the important problems in the theory. 
\vskip.075in

A step in this direction is a recent paper of  Bertucci \cite{Be20} on finite state models which introduced the notion of monotone solutions  for MFG with finite state space and studied its well-posedness.  The work of \cite{Be20}, which is based on a uniqueness technique developed by Lions in \cite{LiCoursCollege},  brought to bear techniques from the theory of viscosity solutions although the actual notion of solution is not related to them. The very  recent work \cite{Be21} by Bertucci extends \cite{Be20} to the continuous state space and for several noise structures, and relies  on a regularity assumption on the solution which is known only for problems with idiosyncratic noise. 
\vskip.075in

%

Here we study  the time-independent master equation without idiosyncratic noise 
which reads as 
\be\label{ME2}
\begin{split}
& U(x,m)-\beta \Delta U(x,m) + H(D_xU(x,m),x)\\[2mm]
&+ \int_{\R^d} D_mU(x,m,y) \cdot D_pH(D_xU(y,m),y)m(dy)\\[2mm]
& -\beta \Bigl( \int_{\R^d} Tr(D^2_{ym}U(x,m,y))m(dy) +2\int_{\R^d} Tr(D^2_{xm} U(x,m,y))m(dy)\\[2mm]
& +\int_{\R^{2d}} Tr(D^2_{mm}U(x,m,y,y'))m(dy)m(dy')\Bigr) = 
F(x,m)  \ \text{in} \  \R^d\times \Pw.
\end{split}
\ee
The unknown is   $U=U(x,m):\R^d\times \Pw\to \R$, where $\Pw$ is the space of Borel probability measures on $\R^d$ with finite second-order moment, $H:\R^d\times \R^d\to \R$ is the Hamiltonian of the problem, $F:\R^d\times \Pw\to \R$ is a continuous map, and $\beta \geq 0$ is the size of the common noise which is assumed to be a $d-$dimensional Brownian motion.   For the meaning of the derivatives of $U$ with respect to $m$ we refer to the books by Cardaliaguet, Delarue, Lasry and Lions  \cite{CDLL} and Carmona and Delarue \cite{CaDeBook}. 
\vs

When $\beta=0$, that is, when there is no common noise,  \eqref{ME2} takes the simpler form 
\be\label{ME1}
\begin{split}
\ds  U(x,m) + H(D_xU(x,m),x)+ 
& \int_{\R^d} D_mU(x, m,y) \cdot D_pH(D_xU(y, m),y)m(dy)\\[3mm] 
& =F(x,m) \ \text{in} \  \R^d\times \Pw,
\end{split}
\ee
and is referred to as the first-order master equation. 
\vs

The solution  $U$ can be interpreted as the value function of a player of a deterministic (when $\beta=0$) or  a stochastic (when $\beta>0$) differential game with infinitely many players whose payoff  is coupled through  $F$.  Notice that the main difference between the first-and second-order equations is that \eqref{ME2} has the additional terms   multiplied by $\beta$, which express the impact of the common noise on the value function $U$ of the small player. 

\vs

The difficult term in \eqref{ME2} and \eqref{ME1} is the nonlocal integral
$$
 \int_{\R^d} D_mU(x, m,y) \cdot D_pH(D_xU(y, m),y)m(dy), 
 $$
which represents  the impact of the crowd of players on a typical small player, makes the equations nonlinear and  infinite dimensional  and   hinders any local  comparison principle and definition. 
 \vs
 
 We work in the so-called monotone setting assuming  that 
 \be
 \label{takis1000}
 H=H(p,x) \ \ \text{ is convex in $p$ \ and \ $F$ \ is monotone in   the Lasry-Lions sense,}
 \ee
 that is, for any $m,m'\in \Pw$,  
$$
\int_{\R^d} (F(x,m)-F(x,m'))(m-m')(dx) \geq 0.
$$
%
%
Without this monotonicity assumption the solution of the master equation might develop discontinuities. The meaning of the solution in this case is an open problem which is completely outside of the scope of the present paper. 
\vs

In contrast, we expect here to have continuous solutions. However, because there is no diffusion term (no idiosyncratic noise),  the solution is, in general,  not smooth. The expected regularity is Lipschitz continuity and  semiconcavity  in space, and continuity in the measure. Hence, the meaning of \eqref{ME2} is, in general,  not clear. Finally, we note that, although the equation contains second derivatives, the common noise is too degenerate to prevent  shocks on the derivative of the solution. 

\vs

To study the second-order master equation we use the Hilbert space approach introduced in \cite{LiCoursCollege} and write \eqref{ME2}  in the Hilbert space $L^2(\Omega; \R^d)$ of $\R^d-$random variables defined on a given probability space $(\Omega, \F, \P)$. Combining ideas from viscosity solutions with \cite{Be20} we define a notion of weak solution of \eqref{ME2} (Definition \ref{def.wealsol}), prove its consistency with the classical formulation \eqref{ME2} when it is smooth (Proposition \ref{takis3.1}), and show that it exists (Theorem \ref{thm.existence}) and is unique up to $m-$dependent constants (Theorem \ref{thm.unique}).

\vs

For  \eqref{ME1}, we also  propose a notion of weak solution directly on $ \R^d\times \Pw$ (Definition \ref{d1}), and  show that it exists (Theorem \ref{thm.existME1}), is  unique again up to $m-$dependent constants (Theorem \ref{thm.uniqueME1}) and consistent (Proposition~\ref{takis3.1}). Finally,  we establish that the two notions of solutions of \eqref{ME1} are equivalent (Theorem \ref{thm.equivalence}). The latter question is reminiscent of similar issues for Hamilton-Jacobi equation in the space of measures as recently investigated by Gangbo and Tudorascu \cite{BaTu19}. 
\vs



Devising a notion of weak solution for \eqref{ME2} turns out to be  much more challenging than for \eqref{ME1}. The reader might bear in mind the analogy with viscosity solutions and the difference in the argument between first- and second-order equation as well as the difficulties due to  the  infinite dimensional set-up. 
\vs


We remark that the notions of weak solution introduced here guarantee that the gradient  (in space) of the solution is unique. 
To eliminate the constant it is necessary to work with the equation satisfied by the solution and not its gradient, which, at the moment, is not possible due to the lack of regularity. Although, for the sake of simplicity, we formulate the results for the master equations \eqref{ME2} and \eqref{ME1}, our work is mainly concerned with the equations satisfied by the derivative $D_xU$ of the value function. All claims could have been written in this set-up, and  we explain this point of view in section \ref{subsec.gradient}. 
\vs

Notice that, in order to mainstream the presentation, we work with the ``stationary'' version of the equations, that is, we have no dependence on time. The extension to time-dependent master equations does not present additional difficulties, although statements are heavier to write and proofs slightly more technical. 
\vs

We continue with a discussion of the general strategy of the paper. The definition of weak solution we introduce here yields that, 
if $U_1,U_2:\R^d\times \mathcal P_2(\R^d)\to \R$ are two continuous in both variables and  Lipschitz continuous with respect to the first variable solutions, then
\be\label{takis1001}
\inf_{m,m'\in \Pw} \int_{\R^d} (U_1(x,m)-U_2(x,m'))(m-m')(dx) \geq 0,
\ee
a fact which, in view of Lemma~\ref{LionsLemma} proven in \cite{LiCoursCollege}, implies that,  for a.e. $x\in \R^d$ and for all $m\in  \Pw$,  
\[D_xU_1(x, m)=D_xU_2(x,m),\] 
and, hence, 
\[U_1(x, m)=U_2(x,m) +c(m) \ \ \text{for some $c\in C(\mathcal P_2(\R^d);\R)$.}\] 


\vs
If $U_1$ and $U_2$ are smooth solutions of  \eqref{ME2} or  \eqref{ME1}, a simple but demanding computation shows that \eqref{takis1001}  is indeed true  in the monotone setting \eqref{takis1000}. 
\vs

 For \eqref{ME1}, this computation relies on writing, for $U=U_1$ or $U=U_2$, the first-order derivative in $m$ of the map 
\be\label{keyquantity}
m\to \int_{\R^d} U(x,m)(m-m')(dx).
\ee
For \eqref{ME2}, it also asks for the second-order derivative in $m$. Of course, if $U$ is not smooth, this computation is unclear. 
\vs 

The breakthrough of \cite{Be20}, in the finite state set-up and of \cite{Be21}, in the continuous space set-up, is to test quantities of the form \eqref{keyquantity} against simple smooth functions, exactly as in viscosity solution theory. In the set-up of \cite{Be20, Be21}, linear test functions are enough. We use variations of this idea in our definitions of weak solutions. 
\vs

For \eqref{ME2} and \eqref{ME1}, there are  three main differences with \cite{Be20}. The first one is that  we work in an infinite dimensional setting. This issue has been already overcome in the framework of viscosity solutions of Hamilton-Jacobi equation by introducing singular test functions; see, for example, Crandall and Lions \cite{CrLi1, CrLi2}, Tataru \cite{Ta82}, Lions \cite{Lions89} and the recent monograph  by Fabbri, Gozzi and \'{S}wi\c{e}ch \cite{FGS17} as well as the references therein. This issue does not appear in \cite{Be21} since the master equation is set in a compact state space (the torus).
\vs
For the first-order master equation 
one can work directly on $\Pw$ and use test functions of the form
$$
\int_{\R^d}\phi(x)m(dx) -\ep {\bf d}_2(m, \tilde m),
$$
with  $\ep>0$,   $\phi:\R^d\to \R^d$ Lipschitz and $\tilde m\in \Pw$, which is the sum of a linear  and  a singular function in $m$. Writing (formally) the equation satisfied by \eqref{keyquantity} and using such class of test functions leads  essentially to our definition of weak solutions for \eqref{ME1}. 
\vs

For the second-order master equation, the argument above  does not work directly because of the second-order terms. This is 
the second  difference with \cite{Be20, Be21}, where the second-order master equations are studied only in a formal way or under a priori regularity conditions on the solution. In the finite dimensional framework, the second order derivatives are handled by the so-called Jensen's Lemma (see, for example,  Crandall, Ishii and Lions \cite{CILuserguide} and the references therein), which has no counterpart in infinite dimension. To deal with this issue we use the  Hilbert space approach to write the master equation in a Hilbert space (see \cite{LiCoursCollege}) and some ideas of the theory of viscosity solutions in infinite dimension put forward in \cite{Lions89} to handle the second-order term. 
\vs

The third and main difficulty compared to \cite{Be20, Be21} is related to the regularity of the solution. Because \eqref{ME2} and \eqref{ME1} contain no idiosyncratic noise (in contrast with the equations studied in \cite{Be21}), the solution is expected to be merely Lipschitz continuous in space. Therefore integrals of quantities of the form $H(D_xU(\cdot, m),\cdot)$ against general probability measures do not make sense.  This requires to introduce a penalization term to the test functions in order to ``touch'' the quantity \eqref{keyquantity} only at measures with a density. This technical point is discussed in details after the definitions of weak solutions. 

\vs

The  MFG theory was introduced by Lasry and Lions in \cite{LLJapan} and, in a particular setting,  by Caines,  Huang and Malham\'{e} \cite{huang2006large}. By now there is a considerable body of literature  in the subject. Listing all the references is beyond the scope of this paper. Early in the development of the theory, it became clear that the ``right object" to study is the master equation, which was introduced in \cite{LiCoursCollege}.  The master equation encompasses all the important properties of the MFG models, and provides the way to obtain approximate Nash equilibria. 
Its analysis has been largely developed by Lions in his courses in Coll\`{e}ge de France \cite{LiCoursCollege}, and then studied, first at a heuristic level by Bensoussan, Frehse and Yam \cite{BeFrYa13} and  Carmona and Delarue \cite{CaDe14}, and  with more rigorous argument by  Gangbo and \'{S}wi\c{e}ch \cite{GaSw15} in the pure first-order case, Buckdahn, Li, Peng and Rainer \cite{BLPR17}, who considered  linear equations with idiosyncratic noise, Chassagneux, Crisan and Delarue \cite{CCD14}, who studied nonlinear equations with  idiosyncratic noise, and Cardaliaguet, Delarue, Lasry and Lions \cite{CDLL} who dealt with nonlinear equations in the presence of  both idiosyncratic and common noises. Since then, many works have been devoted to this topic. Lions also developed the Hilbertian approach \cite{LiCoursCollege} in order to handle equation of the form \eqref{ME2} or \eqref{ME1}, which yields the existence of classical solutions under a structure condition on $H$ and $F$ ensuring the convexity of the solution with respect to the space variable. A partial list of references on the master equation is  \cite{AmMe21, bayraktar2019finite, bertucci2019some, BeLaLi20, Be20, Be21, Bessi20, cardaliaguet2020splitting, CaDeBook, GaMe20, GMMZ21, Ma20, MoZh}. 
\vs

In spite of all the progress mentioned above, an important  question that has remained open is the development of a theory of weak solutions of the master equation,  which is not based on regularity. Indeed, without idiosyncratic noise, the solution is not expected to be more than Lipschitz continuous in the space variable and not more than continuous in the measure variable. The recent papers Gangbo and M\'esz\'{a}ros \cite{GaMe20} and Gangbo, M\'esz\'{a}ros, Mou and  Zhang  \cite{GMMZ21} overcome these difficulties by assuming a structure condition which ensures space convexity and, hence, the smoothness of the solution. First steps in the direction of dealing with nonsmooth  solutions are  the paper of Mou and  Zhang \cite{MoZh}, which discusses some notions of weak solution based on the behavior of the solution with respect to the solution of the mean field game system, as well as the aforementioned works \cite{Be20, Be21}.

\subsection*{Organization of the paper} The paper is organized as follows. In section~1 we introduce the Hilbert space approach and  the basic assumptions. We also state an important technical lemma which is in the background of the uniqueness 
up to a constant. In section 2 we study the  first-order master equation. Section 3 is about the second-order problem. Finally, section 4 discusses the equivalence of the definitions for the first-order master equation.

\subsection*{Notation} Throughout the paper $\mathcal{P}(\R^d)$, $\mathcal{P}_1(\R^d)$ and $\mathcal{P}_2(\R^d)$ are respectively the sets of Borel probability measures on $\R^d$, of Borel probability measures with finite first moment and finite second moment respectively, which are denoted by $M_1$ and $M_2$, that is, given $m\in \mathcal P(\R^d)$,  $M_1(m)=\int_{\R^d} |x|m(dx)$ and $M_2(m)=\int_{\R^d} |x|^2m(dx)$. We let ${\bf d}_1$ and ${\bf d}_2$ be the usual Wasserstein distances on $\mathcal P_1(\R^d)$ and $\Pw$ respectively. We denote  by $\Pw \cap L^\infty(\R^d)$ 
the set of measures $m\in \Pw$ which are absolutely continuous with respect to the Lebesgue measure in $\R^d$ with density in $L^\infty(\R^d)$, which we also denote by $m$. If $h:\R^d\to \R^d$ is a Borel measurable map and $m\in \mathcal P(\R^d)$, we write $h\sharp m$ for the image by $h$ of the measure $m$. 
If $f\in L^\infty(\mathcal O)$, then $\|f\|_{\mathcal O, \infty}$ is  the usual $L^\infty$-norm. When $\mathcal O=\R^d$, then we simply write $\|f\|_\infty.$  The inner product between $x, y \in \R^d$ is $x\cdot y$.  Finally, given $m:\mathcal O \to \R_+$ Borel-measurable, $L^2_m(\mathcal O;\R^k)= \{f:\mathcal O\to \R^k: \int_{\mathcal O} |f(x)|^2 m(x) dx<\infty \}.$ When the domain is $\R^d$, we simply write $L^2_m$.

\section{Preliminaries and assumptions}

We recall several facts about the notion of monotonicity,  the Hilbert space approach, some notation from the theory of viscosity solutions in Hilbert spaces, and state the main assumptions. 

\subsection*{A key lemma on monotonicity} Following  \cite{LLJapan, LiCoursCollege}, the notion of monotonicity plays a central role in the analysis of the master equations \eqref{ME2} and \eqref{ME1}. This can be illustrated by the following lemma, proven in  \cite{LiCoursCollege}, which links monotonicity with  uniqueness and  plays an instrumental role in the paper.

\begin{lem}\label{LionsLemma} Assume that $U_1,U_2:\R^d\times \mathcal P_2(\R^d)\to \R$ are continuous in both variables and  Lipschitz continuous with respect to the first variable, and that, for all  $m,m'\in \Pw$,
\be\label{cond.mono}
\int_{\R^d} (U_1(x,m)-U_2(x,m'))(m-m')(dx)\geq 0. 
\ee
Then, for a.e. $x$ and all $m\in  \Pw$,  $D_xU_1(x, m)=D_xU_2(x,m)$. 
\end{lem}

\begin{proof}
Fix $m_0\in \Pw$ and  $\bar x\in \R^d$ and, for $h\in (0,1)$,   let $m'= (1-h)m_0+h\delta_{\bar x}$, where 
$\delta_{\bar x}$ denotes the $\delta$ mass at $\bar x$. 
\vs

It follows from the assumption that 
$$
\int_{\R^d} (U_1(x,m_0)-U_2(x,(1-h)m_0+h\delta_{\bar x}))(m_0-\delta_{\bar x})(dx)\geq 0.
$$
In view of the continuity of $U_1$ and $U_2$, letting $h\to 0^+$ we get   
$$
\int_{\R^d} (U_1(x,m_0)-U_2(x,m_0))(m_0-\delta_{\bar x})(dx)\geq 0,
$$
which can be rewritten as
$$
U_1(\bar x,m_0)- U_2(\bar x, m_0) \leq  \int_{\R^d} (U_1(x,m_0)-U_2(x,m_0))m_0(dx).  
$$
Since a similar argument yields the reverse inequality, we find that for all $\bar x\in \R^d$, 
$$
U_1(\bar x,m_0)- U_2(\bar x, m_0) = \int_{\R^d} (U_1(x,m_0)-U_2(x,m_0))m_0(dx).  
$$
Hence,  $U_1(\cdot,m_0)-U_2(\cdot,m_0)$ is constant and therefore $D_xU_1(\cdot,m_0)=D_xU_2(\cdot,m_0)$ a.e.. 

\end{proof}

\subsection*{The Hilbert space approach}\label{subsec.Hsp} 

In order to investigate a notion of weak solution of \eqref{ME2}, we follow \cite{LiCoursCollege} and formulate the equation in the space $L^2(\Omega; \R^d)$ of $\R^d-$valued random variables, where  $(\Omega, \mathcal F, \P)$ is an atomless probability space. We write $L^2$ for $L^2(\Omega;\R^d)$, $\E$ for  expectation, and $\mathcal L(X)$ for the law of the random variable $X$. 
\vs

If $U:\Pw\to \R$ and $X\in L^2$, we set $\tilde U(X)=U(\mathcal L(X))$. It turns out (see \cite{LiCoursCollege, CaDeBook}) that $U$ is continuous if and only if $\tilde U$ is continuous. In addition, $U$ is differentiable at $m\in \Pw$ if and only if $\tilde U$ is Frechet differentiable at some (and then all) random variable $X$ such that $\mathcal L(X)=m$ and 
$$
D_X\tilde U(X)= D_mU(X, m). 
$$

To handle the terms related with the common noise in \eqref{ME2},  one has to keep in mind that they  are the impact of the common noise on the value function. In other words, if $W$ is a $d-$dimensional Brownian motion defined on a different  probability space $(\Omega', \mathcal F', \P')$ with expectation denoted by $\E'$ and if $U=U(x,m)$ is a sufficiently smooth map, then (see, for example,  \cite{LiCoursCollege, CDLL, CaDeBook})
\begin{align*}
& \E'\left[ U(x+\sqrt{2\beta} W_t, (Id+\sqrt{2\beta} W_t)\sharp m)\right] - U(x,m)\\
& =  \beta t \Bigl( \Delta U(x,m) +\int_{\R^d} Tr(D^2_{ym}U(x,m,y))m(dy) +2\int_{\R^d} Tr(D^2_{xm} U(x,m,y))m(dy)\\[2mm]
& \qquad +\int_{\R^{2d}} Tr(D^2_{mm}U(x,m,y,y'))m(dy)m(dy')\Bigr)+o(t) .
\end{align*}
If  $\tilde U(x,X)= U(x, \mathcal L(X))$ and $e_1,\ldots, e_d$ is the  canonical basis of $\R^d$,  then, since  $W$ is independent of $X$, the equality above becomes 
\begin{align*}
& \E'\left[ \tilde U(x+\sqrt{2\beta} W_t, X+\sqrt{2\beta} W_t)\right] - U(x,m)\\
& =  \beta t \Bigl( \sum_{k=1}^d (D^2_{xx}\tilde U(x,X)+ 2D^2_{xX}\tilde U(x,X)+D^2_{XX}(x,X))(e_k,e_k)\Bigr)+o(t).
\end{align*}
\vs

With this in mind, the equation \eqref{ME2} written on $L^2$ takes the form 
\be\label{eq.tildeU} 
\begin{split}
 \ds \tilde U(x, X)+H(D_x \tilde U(x,X), x) + \E\left[ D_X\tilde U(x,X)\cdot D_pH(D_x\tilde U(X,X),X)\right] \\
 \ds -\beta\Big( \sum_{k=1}^d (D^2_{xx}\tilde U(x,X)+ 2D^2_{xX}\tilde U(x,X)+D^2_{XX}\tilde U(x,X))(e_k,e_k)\Big)\\
\hskip2in  =F(x,\mathcal L(X)) \   {\rm in }\ \R^d\times L^2.
\end{split}
\ee
\vs

\subsection*{Tools from the theory of viscosity solutions in infinite dimensions} \label{subsec.Tools}

As discussed in the introduction,  to define a notion of weak solution of \eqref{ME2} we need to manipulate quantities of the form \eqref{keyquantity}. It is actually even more convenient to also  relax the variable $\tilde m$ in \eqref{keyquantity} and, using the Hilbert space approach, to look  at the map 
$$
(X,Y)\to \hat U(X,Y)=\E\left[ \tilde U(X, \mathcal L(X))-\tilde U(Y, \mathcal L(X))\right]. 
$$
The ``equation'' satisfied by $\hat U$ follows from  \eqref{eq.tildeU} and contains many terms. Here we  only discuss the second-order term (the one multiplied by $\beta$) in \eqref{eq.tildeU}. It is given by 
\begin{align*}
& \mathcal L(X,Y)= \E\Bigl[ \sum_{k=1}^d (D^2_{xx}(\tilde U(X,X)-\tilde U(Y,X))+ 2D^2_{xX}(\tilde U(X,X)-\tilde U(Y,X))\\
& \qquad +D^2_{XX}(\tilde U(X,X)-\tilde U(Y,X)))(e_k,e_k)\Bigr]. 
\end{align*}
We note, after computing the second-order derivative of $\hat U$, that 
\begin{align*}
\mathcal L(X,Y)= \beta \sum_{k=1}^d D^2_{(X,Y)}\hat U(X,Y)((e_k,e_k),(e_k,e_k)). 
\end{align*}

This leads to the introduction a particular second-order operators on $L^2\times L^2$ as follows. If $\mathcal X$ is a bilinear form on $L^2\times L^2$, we set 
\be\label{def.Lambda}
\Lambda(\mathcal X)= \sum_{k=1}^d \mathcal X((e_k,e_k),(e_k,e_k)). 
\ee
Here we use the fact that, since each $e_k$ can be seen as a constant random variable on $\R^d$. $e_k$ is also  an element of $L^2$. 
\vs 

It is immediate that $-\Lambda$ is a degenerate elliptic operator and satisfies condition (2) and (6) in \cite{Lions89}.  Indeed,  since $\Lambda$ is linear, if $H_N$ is an increasing sequence of finite-dimensional subspaces of $H=L^2\times L^2$ and $P_N$ and $Q_N$ are the projections onto $H_N$ and $H_N^\perp$ respectively, condition (6) of  \cite{Lions89} can be written as 
$$
\lim_{N\to +\infty} \left\{ | \Lambda (Q_N) |\right\} =0.
$$
\vs

In fact, if $H_N$ contains all constant random variables, the space of which  has  of dimension $2d$, then we actually have $\Lambda(Q_N)=0$.  From  now on we fix $H_N$ with this property. 
\vs

For completeness, following \cite{Lions89},  we recall next the notions of   the second-order subdifferential $D^{2,-}$ and second-order subjet ${\overline D}^{2,-}$ of a map from $L^2\times L^2\to \R$. 
\vs

To simplify the notation, we consider  a lower semicontinuous map $\phi:H\to \R$, where $H$ is a general separable  Hilbert space, write $L'(H)$ for the space of bounded bilinear forms on $H$ and denote the $H$-inner product by $<\cdot,\cdot>$. In the problem 
 we are studying here,  $H=L^2\times L^2$ and $x=(X,Y)$.
\vs

Given $x_0\in H$ and $\phi :H \to \R$ lower semicontinuous, we say that $L'(H) \times H  {\ni} (X,p) \in D^{2,-} \phi (x_0)$ if 

\[\underset{x\to x_0}\liminf \Big[\|x-x_o\|^{-2} \Big(\phi (x)-\phi(x_0) - 
( p,x-x_0) -  \frac{1}{2} (X(x-x_0),x-x_0)\Big)\Big] \geq 0.\] 
It turns out that $(X,p) \in D^{2,-} \phi (x_0)$ is equivalent to the existence of $\psi \in C^2(H;\R)$ such that the map $x\to \phi (x)-\psi (x)$ attains  a  minimum at $x_0$ and $(X,p)=(D^2\psi (x_0), D\psi(x_0))$.  Since we are working a separable Hilbert space, this last fact is proved as in the finite dimensional case.
\vs  

Finally,
\[\begin{split}
  {\overline D}^{2,-}\phi(x)&=\Big \{(X,p) \in L'(H) \times H: \ \text{there exist} \  (X_n,p_n,x_n) \in L'(H) \times H \times H  \ \text{such that} \\
 &  (X_n,p_n)\in D^{2,-} \phi (x_n) \ \ \text{and} \ \ \ (x_n, p_n, X_n, \phi(x_n)) \underset{n\to \infty}\to (x,p, X,\phi(x))\Big\}.
\end{split}\]
\vs

In section 4 we will also need to refer to the first-order superdifferential $D^+\psi(x_0)$ of an upper semicontinuous $\psi:H\to \R$ which is the possibly empty set of $p\in H$ such that 
$$\underset{x\to x_0}\limsup \Big[\|x-x_o\|^{-1} \Big(\psi (x)-\phi(x_0) - 
( p,x-x_0) \Big)\Big] \leq 0.$$

\subsection*{The assumptions} We complete the introduction by stating some of the assumptions needed in our study. 
\vs
Throughout the paper we assume that 
\be\label{takis1002}
 H \in C^1(\R^d\times \R^d;\R), \  H(0,x) \ \text{is  bounded and $H$ is convex in the first variable,}
\ee
and
\be\label{takis1003} 
F:\R^d\times \mathcal P_1(\R^d)\to \R \ \text{ is Lipschitz continuous, monotone and bounded.}  
\ee
\vs

For the existence proof we will need much stronger  conditions. In particular,  we assume  that 
\be\label{HH}
\begin{cases}
(i) \; \text{for any $R>0$,  there exists $C_R>0$ such that,}\\[2mm]
\text{ for all $x,p\in \R^d$ with  $|p|\leq R$, }\\[2mm]
\qquad  |H (p,x)|+ |D_p H(p,x)|+ |D^2_{px}H(p,x)|+ |D^2_{pp}H(p,x)|\leq C_R, \\[2mm]
(ii)\; \text{ there exists $\lambda >0$ and $C_0>$ such that,}\\[2mm] 
\text{for any $p,q,x,z\in \R^d$ with $|z|=1$ and in the sense of distributions, }\\[2mm]
\lambda(D_pH(p,x)\cdot p- H(p,x))+ D^2_{pp}H(p,x)q\cdot q \\[1.5mm]
  \qquad \qquad \qquad \qquad \qquad \qquad  + 2D^2_{px}H(p,x)z\cdot q + D^2_{xx}H(p,x) \ z\cdot z \geq - C_0,
\end{cases}
\ee
and
\be\label{FG}
\begin{cases}
 F\in C(\R^d\times \mathcal P_1(\R^d);\R) \ \text{and 
there exists $C_0>0$ such that}\\[2mm]
\underset{m\in \mathcal P_1(\R^d), \; t\in [0,T]}\sup\left[
\|F(\cdot, m)\|_\infty +\|DF(\cdot ,m)\|_\infty + \|D^2F(\cdot ,m)\|_\infty\right] \leq C_0.
\end{cases}
\ee

We also need to reinforce the monotonicity condition on $F$ by assuming that 
\be\label{takis22}
\begin{cases}
\text{ 
there exits $\alpha_F>0$ such that, for all  $m_1,m_2\in \mathcal P_1(\R^d)$},\\[2mm] 
\ds \int_{\R^d} ( F(x,m_1)- F(x,m_2)) (m_1-m_2)(dx) \geq \alpha_F\int_{\R^d} (F(x,m_1)- F(x,m_2))^2dx,\\[2mm] 
\end{cases}
\ee
and 

\begin{equation}\label{SM2}
\displaystyle \int_{\R^d} (F (x,m_1)-F(x,m_2)) (m_1-m_2)(dx)\geq 0 \ \  \text{implies} \ \ m_1=m_2.
\end{equation}

Conditions \eqref{HH} and \eqref{FG}  ensure that the solution of the master equation is bounded and uniformly semiconcave. The strong monotonicity condition \eqref{takis22}, the strict monotonicity condition \eqref{SM2}, as well as \eqref{HH} and \eqref{FG}   were used by the authors in  \cite{CaSo20} to solve the underlying backward-forward system of stochastic  partial differential equations. The results of \cite{CaSo20} are  used here to establish the existence of the weak solution solution of \eqref{ME2}. 

\section{The first-order master equation}

\subsection*{The notion of weak solution} We study the first-order master equation \eqref{ME1} in $\R^d\times \Pw$ and introduce the following definition of weak solution.

\begin{defn}\label{d1} A bounded function $U=U(x,m) \in C(\R^d\times \Pw) $ is a weak solution of \eqref{ME1}, if $U$ is Lipschitz continuous and semiconcave in the first  variable both uniformly in the second variable, and there exists $C>0$ such that, for any $\tilde m\in \Pw\cap L^\infty(\R^d)$,  $\hat m\in \Pw$, $\phi\in W^{1,\infty}(\R^d)$ and $\ep>0$ such that the map 
$$
m \to \int_{\R^d} (U(x,m)-\phi(x))(m(x)-\tilde m(x))dx +\ep(\dw(m, \hat m)+ \|m\|_\infty)
$$
has a local minimum at $m_0$ in $ \Pw\cap L^\infty(\R^d))$, we have 
\begin{align}\label{Cond1}
& \ds \int_{\R^d} U(x, m) (m_0(x)-\tilde m(x))dx + \int_{\R^d} H(D_xU(x, m_0),x)(m_0(x)-\tilde m(x))dx\notag \\
& \qquad \ds - \int_{\R^d} (D_xU(y, m_0)-D\phi(y))\cdot D_pH(D_xU(y,m),y )m_0(y)dy\\
& \qquad\qquad \ds \geq 
\int_{\R^d} F(x,m_0)(m_0(x)-\tilde m(x))dx- C\ep(1+\|m_0\|_{\infty}). \notag
\end{align}
\end{defn}
The idea of using this argument to define weak solution of \eqref{ME1} goes back to  \cite{Be20}, which considered  a finite state space model. 
\vs

We continue with some remarks on the notion of weak solution.  
\vs
First notice that, if $U$ as in Definition~\ref{d1} and $c\in C(\Pw)$, then $U+c$ is also a weak solution, since
\[\begin{split}
&\ds \int_{\R^d} ( U(x,m) +c(m) -\phi(x))(m(x)-\tilde m(x))dx\\
&\ds =\int_{\R^d} ( U(x,m)-\phi(x))(m(x)-\tilde m(x))dx + \int_{\R^d} c(m) (m(x)-\tilde m(x))dx\\
&\ds =\int_{\R^d} ( U(x,m) -\phi(x))(m(x)-\tilde m(x))dx.\end{split}\]
So the definition is more a notion of weak solution for $D_xU$ than for $U$. We develop this point for the second-order master equation at the end of section \ref{subsec.gradient}. 
\vs
The heuristic explanation of the definition is as follows.  Ignoring the penalization terms in $\ep$,  that is, assuming that  $\ep=0$, and assuming that  $U$ is  a smooth solution of \eqref{ME1}, we see that, if the map
$$
m \to \int_{\R^d} (U(x,m)-\phi(x))(m(x)-\tilde m(x))dx
$$
has a local minimum at $m_0$, then the first-order optimality condition yields 
\be\label{ilzakjensrd}
\int_{\R^d} D_mU(x,m,y)(m(x)-\tilde m(x))dx + D_xU(y,m)-D\phi(y)=0. 
\ee
On the other hand, integrating \eqref{ME1} against $(m-\tilde m)$, we find
\[\begin{split}
  \int_{\R^d} (U(x,m) +  & H(D_xU(x,m),x))(m(x)-\tilde m(x))dx\\
& +  \int_{\R^{2d}} D_mU(x, m,y) \cdot D_pH(D_xU(y, m),y)m(dy)(m-\tilde m)(dx)\\
& \hskip.5in   =\int_{\R^d} F(x,m)(m(x)-\tilde m(x))dx. 
\end{split}\]
Using \eqref{ilzakjensrd} in the second term gives
\begin{align*}
&\ds  \int_{\R^d} (U(x,m) + H(D_xU(x,m),x))(m(x)-\tilde m(x))dx\\[1.5mm]
& +  \int_{\R^{d}} (D_xU(x,m,y)-D\phi(y)) \cdot D_pH(D_xU(y, m),y)m(dy)  =\int_{\R^d} F(x,m)(m(x)-\tilde m(x))dx,
\end{align*}
which is precisely \eqref{Cond1} up to the penalization terms in $\ep$. 
\vs

There are some important differences between \cite{Be20, Be21} and our setting which is infinite dimensional. They are  
the lack of local compactness of the state space, the nonlocality  of  the nonlinearity and the low, only Lipschitz continuity, regularity of the solution. 
\vs

As in previous works for Hamilton-Jacobi equations in infinite dimensions, see, for example, \cite{Ta82},  we deal with the lack of compactness by introducing the penalization term $\ep {\bf d}_2(m,\tilde m)$ in the definition. 
\vs

Recall that the  solution $U=U(x,m)$  is only almost everywhere differentiable in $x$. As a result, the nonlocal transport term makes sense only if the integral is against absolutely continuous measures with enough integrability. As we will see later, this is also related with 
the construction of a solution, for which there is a natural representation formula only when the measure is absolutely continuous with a bounded density. These consideration leads us to add the penalization term $\ep \|m\|_{L^\infty}$.

\subsection*{The existence of weak solutions}   

To prove the existence,  we consider, 
for $t_0\geq 0$ and   $m_0\in \Pw \cap L^\infty(\R^d)$, the solution  $(u,m)$ 
of the classical forward-backward MFG system
\be\label{takis2.2}
\left\{\begin{array}{l}
\partial_t u = u+H(Du,x) - F(x,m) \ \text{in} \ \R^d \times (t_0,\infty), \\[1.5mm] 
\partial_t m ={\rm div}(m D_pH(Du,x))  \ \text{in} \ \R^d \times (t_0,\infty),\\[1.5mm]
m(\cdot, t_0)=m_0, 
\end{array}\right.
\ee
whose existence and uniqueness  follows from   \cite{LLJapan}. Recall that a solution of \eqref{takis2.2} is  a pair $(u,m):\R^d\times [0,+\infty)\to \R\times [0,+\infty)$ such that $u$ is a Lipschitz continuous and semiconcave in space viscosity solution of the first equation while $m$ is a bounded solution of the second equation in the sense of distribution; see \cite{LLJapan} and \cite{CaHa} for details. 
\vs

Since $H$ and $F$ do not depend on time, the uniqueness of the solution of \eqref{takis2.2} implies that  $u(\cdot, t_0)$ is independent of $t_0$. 
\vs
The candidate weak solution of \eqref{ME1}  is 
\be\label{takis2.1}
U(x,m_0)= u(x, t_0),
\ee
which, as the next theorem asserts, is a weak solution of \eqref{ME1} on $\R^d\times \Pw.$

\begin{thm}\label{thm.existME1} Assume \eqref{HH}, \eqref{FG}, \eqref{takis22} and \eqref{SM2}. Then the map $U=U(x,m): \R^d\times (\Pw \cap L^\infty(\R^d))$ defined by \eqref{takis2.1} has a continuous extension on $\R^d \times  \Pw$, which is   a weak solution of \eqref{ME1}.
\end{thm}

\begin{proof} The extension property and the regularity part in the definition of weak solutions are standard and we omit their proof (the extension is explained in details in the second-order case). Here we only check the latter part of Definition~\ref{d1}. 
\vs 

For the argument we need the following lemma. For its proof we refer to \cite{LLJapan} and \cite{CaHa}.

\begin{lem}\label{lem.constC} There exists a $C>0$, which is  independent of $t_0$ and $m_0$, such that, for all $t\in (t_0,t_0+1),$ 
\begin{align*}
&\dw(m(t),m_0)\leq C (t-t_0), \; \|m(t)\|_{\infty}\leq (1+C(t-t_0)) \|m_0\|_{\infty} \ \  \text{and} \\
&\qquad   \ M_2(m(t))\leq (1+C(t-t_0))M_2(m_0).
\end{align*}
\end{lem}
\vs

Let $\tilde m \in (\Pw\cap L^\infty(\R^d))$, $\hat m\in \Pw$,  $\phi\in W^{1,\infty}(\R^d)$ and $\ep>0$ be such that the map 
$$
m\to \int_{\R^d} ( U(x,m)-\phi(x))(m(x)-\tilde m(x))dx +\ep (\dw(m,\hat m)+\|m\|_\infty)
$$
has a local minimum at $m_0\in  (\Pw\cap L^\infty(\R^d))$. 
\vs

Fix some $t_0\geq 0$, and consider the solution $(u,m)$ of \eqref{takis2.2} with initial condition $m_0$. 
\vs

Then, for $h>0$ small, we have  
\begin{align*}
&\int_{\R^d} ( U(x,m(t_0+h))-\phi(x))(m(x, t_0+h)-\tilde m(x))dx \\
& \qquad \qquad +\ep(\dw(m(t_0+h), \hat m)+\|m(t_0+h)\|_{L^\infty(\R^d)})) \geq \\
& \int_{\R^d} ( U(x,m(t_0))-\phi(x))(m(x, t_0)-\tilde m(x))dx +\ep(\dw(m_0, \hat m)+\|m_0\|_\infty).
\end{align*}
Using \eqref{takis2.1} and  Lemma \ref{lem.constC} we find
\begin{align*}
&\int_{\R^d} ( u(x, t_0+h)-\phi(x))(m(x, t_0+h)-\tilde m(x))dx\\
& \qquad + \ep (\dw(m_0, \hat m)+Ch) +\ep (1+Ch) (\|m_0\|_{L^\infty(\R^d)})\\
&\geq \int_{\R^d} ( u(x, t_0)-\phi(x))(m(x, t_0)-\tilde m(x))dx +\ep (\dw(m_0, \hat m)+\|m_0\|_\infty). 
\end{align*}
The classical, in the MFG-context, identity  
\begin{align*}
& \frac{d}{dt}\int_{\R^d} u(x,t)m(x,t)dx =\\
& \int_{\R^d} (u + H(Du(x,t),x,t)-D_pH(Du(x,t),x,t)\cdot Du(x,t)-F(x,m(t))m(x,t)dx, 
\end{align*}
and the equation for $u$ and $m$ yield that 
\begin{align*}
& \int_{t_0}^{t_0+h} \int_{\R^d} ( u+ H(Du(x,t),x,t) -F(x,m(t))(m(x,t)-\tilde m(x))dxdt\\\
& - \int_{t_0}^{t_0+h} \int_{\R^d} D_pH(Du(x,t),x,t)\cdot (Du(x,t)-D\phi(x)) m(x,t)dxdt\\ 
& \qquad \geq -C\ep h-\ep Ch\|m_0\|_\infty. 
\end{align*}
Dividing  by $h$ and letting  $h\to 0$ we obtain the result, since $Du(x, t_0)=D_xU(x,m_0)$ and $u$ is uniformly semiconcave in space while $m$ is bounded in $L^\infty$, has a bounded second-order moment  and is $L^\infty-$weak $\star$ continuous in time.  

\end{proof}

\subsection*{The uniqueness of the weak solution} 

We continue with the uniqueness result about weak solution. In view of the observation in Lemma~\ref{LionsLemma}, 
we actually prove  uniqueness up to an $m$-dependent constant. Hence,  the gradient in $x$ of a weak solution is unique.

\begin{thm}\label{thm.uniqueME1} Assume \eqref{takis1002} and \eqref{takis1003}. Then weak solutions of \eqref{ME1} are unique up to an $m$-dependent constant.
\end{thm}

\begin{proof} 
Assume that $U$ and $\tilde U$ are two weak solutions of \eqref{ME1}.  Arguing along the lines of \cite{Be20, Be21}, the key point is  to prove that
\be\label{pierre}
M=\inf_{(m , \tilde m)\in \Pw\times \Pw}  \int_{\R^d} (U(x, m)-\tilde U(x,\tilde m) )(m(x)-\tilde m(x))dx \ \geq 0.
\ee
Then the conclusion follows from Lemma~\ref{LionsLemma}. 
\vs
The proof of \eqref{pierre} is achieved by penalization.  Fix  $\ep>0$ small and consider the map
\begin{equation*}
\begin{split}
\Phi_\ep(m,\tilde m)&  =  \int_{\R^d} (U(x, m)-\tilde U(x,\tilde m) )(m(x)-\tilde m(x))dx \\ 
&  \qquad +\ep\left( \|m\|_\infty+ \|\tilde m\|_{\infty}\right),
\end{split}
\end{equation*}
which  is lower semicontinuous and bounded from  below on $\Pw\times \Pw$ with values in $\R\cup\{+\infty\}$. 
\vs

Next fix  some $\hat m\in \Pw\cap L^\infty$. It follows from  Ekeland's variational principle \cite{Ek74} that there exists a minimum $ m_\ep, \tilde m_\ep$ of the map 
\be\label{takis2.3}
(m,\tilde m) \to \Phi_\ep(m,\tilde m)+\ep(\dw(m, \hat m)+ \dw(\tilde m, \hat m)).
\ee
Classical arguments  show that 
$$
\lim_{\ep\to0} \Phi_\ep(m_\ep,\tilde m_\ep)+\ep(\dw(m_\ep, \hat m)+ \dw(\tilde m_\ep, \hat m)) =M
$$
and, therefore, 
$$
\lim_{\ep\to 0^+} \ep\left( \|m_\ep\|_{\infty}+ \|\tilde m_\ep\|_{\infty}+\dw(m_\ep, \hat m)+ \dw(\tilde m_\ep, \hat m)\right) =0.
$$
Using \eqref{takis2.3} and the fact that $U$ and $\tilde U$ are weak solutions, we find 
\begin{align*}
& \ds \int_{R^d} U(x,m)(m_\ep(x) -\tilde m_\ep(x))dx  +\int_{\R^d} H(D_xU(x,m_\ep),x)(m_\ep(x)-\tilde m_\ep(x))dx \\
& \qquad \ds - \int_{\R^d} (D_xU(y,m_\ep)-D_x\tilde U(y, \tilde m_\ep))\cdot D_pH(D_xU(y, m_\ep),y)m_\ep(y)dy\\
& \qquad\qquad \ds \geq 
\int_{\R^d} F(x,m_\ep)(m_\ep(x)-\tilde m_\ep(x))dx- C\ep(1+\|m_\ep\|_{L^\infty(\R^d)})
\end{align*}
and
\begin{align*}
& \ds \int_{R^d} \tilde U(x,\tilde m_\ep)(\tilde m_\ep(x) -m_\ep(x))dx +\int_{\R^d} H(D_x\tilde U(x, \tilde m_\ep), x )(\tilde m_\ep(x)- m_\ep(x))dx \\
& \qquad \ds - \int_{\R^d} (D_x\tilde U(y, \tilde m_\ep)-D_x U(y,m_\ep))\cdot D_pH(D_x\tilde U(y, \tilde m_\ep),y t_\ep)\tilde m_\ep(y)dy\\
& \qquad\qquad \ds \geq 
\int_{\R^d} F(x,\tilde m_\ep)(\tilde m_\ep(x)- m_\ep(x))dx- C\ep(1+\|\tilde m_\ep\|_{\infty}).
\end{align*}
Adding the last two  inequalities we find 
\begin{align*}
& \ds \int_{\R^d} \left(U(x,m_\ep)-\tilde U(x,\tilde m_\ep)\right)(m_\ep(x) -\tilde m_\ep(x))dx
+ C\ep(1+\|m_\ep\|_{\infty}+\|\tilde m_\ep\|_{\infty})\\ 
& \ds \geq \int_{\R^d} (H(D_x\tilde U(x,\tilde m_\ep),x)-H(D_x U(x,m_\ep),x)\\
& \ds \qquad \qquad -D_pH(D_xU(x, m_\ep))\cdot (D_x\tilde U(x,\tilde m_\ep)-D_x U(x, m_\ep))  m_\ep(x)dx \\ 
& \ds \qquad+ \int_{\R^d} (H(D_xU(x,m_\ep),x)-H(D_x \tilde U(x,\tilde m_\ep),x)\\
& \ds \qquad \qquad -D_pH(D_x\tilde U(x, \tilde m_\ep),x)\cdot (D_x U(x, m_\ep)-D_x \tilde U(x, \tilde m_\ep))) \tilde  m_\ep(x)dx \\ 
& \ds\qquad+ \int_{\R^d} (F(x,m_\ep)-F(x,\tilde m_\ep))(m_\ep(x)-\tilde m_\ep(x))dx.
\end{align*}
In view of the convexity of $H$ in the gradient argument and the monotonicity of  $F$, the right-hand side of the inequality above  is nonnegative.
Hence, letting $\ep\to0$ leads to $M\geq 0$, which is the desired result.  

\end{proof} 

\section{The second-order master equation with common noise only}

\vs

The definition of weak solutions of \eqref{ME2} and its analysis are considerably more involved than the one of \eqref{ME1} due to the presence of the extra terms arising from the common noise. Although we are not dealing with viscosity solutions, readers  should draw  of the 
analogy and level of complications in the theory of first- and second-order viscosity solutions. One of the main reasons, is the need to deal with second derivatives in $m$.  For this, we consider the Hilbert space formulation of the master equation, which was introduced in \cite{LiCoursCollege}.

\vs

We also remark that, although the weak solution we introduce is not a viscosity solution of the master equation,  the arguments used to obtain the uniqueness use several steps of the proof of the uniqueness of viscosity solutions. This similarity can be already seen in \cite{Be20, Be21}. 

\vs 


The rest of the section is divided  in three subsections. The first is about the Hilbert space formulation, the definition of weak solution,
and the consistency with classical solutions. The second is about the uniqueness and the third is about the existence.

\subsection*{The notion of weak solution} In order to define the notion of weak solution for \eqref{ME2}, we  recall some notation from the Hilbert space approach discussed in section~1. 
In addition to the general setting and notation already discussed there,  we also consider the set  $L^\infty_{ac}$  of random variables $X\in L^2=L^2(\Omega; \R^d)$ such that $\mathcal L(X)$ is absolutely continuous with respect to the Lebesgue measure $\lambda$ on $\R^d$ and such that $d\mathcal L(X)/d\lambda \in L^\infty(\R^d)$. The operator $\Lambda$ is defined in \eqref{def.Lambda}.
\vs 

Given  $U\in C(\R^d\times \Pw;\R)$, the map  $\hat U :L^2\times L^2\to \R$ is defined, for all $(X,Y)\in L^2\times L^2$, by 
$$
\hat U(X,Y)= \E\left[ U(X, \mathcal L(X))-U(Y, \mathcal L(X))\right].  
$$
For any $\ep>0$, we also consider $\hat U^\ep:L^\infty_{ac}\times L^2\to \R$ given by 
$$
\hat U^\ep(X,Y)=  \E\left[ U(X, \mathcal L(X))-U(Y, \mathcal L(X))\right] + \ep  \left\|\frac{d\mathcal L(X)}{d\lambda}\right\|_{\infty}. 
$$
Notice that $\hat U$ and $\hat U^\ep$ actually only depend on $D_xU$, if  $D_xU$ exists,  and not on $U$. 
\vs

The notion of weak solution is introduced next. 
\begin{defn}\label{def.wealsol} A bounded map $U\in C(\R^d\times \mathcal P_2(\R^d))$ is a weak solution of the master equation \eqref{ME2}  if  $U$ is Lipschitz continuous and semiconcave in $x$ uniformly  in $m$ and there exists a constant $C>0$ such that, for all $( X, Y)\in L^\infty_{ac}\times L^2$, any $\ep\in (0,1)$ and all 
$(\mathcal X, (p_X,p_Y))\in  \bar D^{2,-}\hat U^\ep(X,Y)$, 
\be\label{cond2BIS}
\begin{split}
&  0 \leq  \hat U( X,  Y) -\beta \Lambda \left( \mathcal X \right)
 - \E\left[F( X, \mathcal L( X))-F( Y, \mathcal L( X))\right]\\
& +\E\left[ H(  D_xU( X, \mathcal L( X)),X)-H(-p_Y,Y)\right]\\
& - \E \left[ (D_xU( X, \mathcal L( X))-p_X)\cdot D_pH( D_xU( X, \mathcal L( X)),X)\right] +C\ep\left(1+ \left\|\frac{d\mathcal L(X)}{d\lambda}\right\|_{\infty} \right).
\end{split}
\ee
\end{defn}

Following Definition~\ref{def.wealsol}, it is necessary to make a number of remarks.  
\vs
As it will be apparent below, the definition actually characterizes $W=D_xU$ and not $U$. Characterizing $U$ seems to be a much harder problem. Indeed, given $D_xU$,  \eqref{ME2} becomes a linear transport equation in the space of measures with a drift $D_pH(D_xU)$ which has a poor regularity. 
\vs
The assumptions made on $U$  can be translated into assumptions in $W=D_xU$: we will do so at the end of the section. 
\vs

The penalization term $\left\|\frac{d\mathcal L(X)}{d\lambda}\right\|_{\infty}$ is needed to ensure that, since $D_xU$ is only defined a.e. in $\R^d$, the terms 
\[
\E\left[ H(  D_xU( X, \mathcal L( X)),X)\right] \ \ \text{and} \ \  \E \left[ (D_xU( X, \mathcal L( X))-p_X)\cdot D_pH( D_xU( X, \mathcal L( X)),X)\right]. 
\]
are well defined. 
\vs 
Finally, we note that  the probability space $(\Omega, \mathcal F, \P)$  is fixed. It is intuitively clear that the notion of solution should not depend on the particular choice of the probability space, but we do not check this here.

\subsection*{Consistency of the definition} 
The following proposition is about the consistency of the notion of weak solution we consider here.

\begin{prop}\label{takis3.1} Assume  that $H\in C^1(\R^d\times \R^d;\R)$ and $F \in C(\R^d\times \mathcal P_1(\R^d); \R)$.
If $U:\R^d\times \Pw\to \R$ is a weak solution of \eqref{ME2} and  $U$, $D_xU, D_mU, D_{xx}, D_{mm}U$ and $D_{xm}U$  
are continuous in $x$ and $m$, then $U$ is a classical solution of  \eqref{ME2} up to adding a continuous function of $m$ to the right-hand side of   \eqref{ME2}. 
\end{prop} 

Here again one can see that the notion of weak solution characterizes the space derivative of $U$. This point is developed in more detail later in this section. 

\begin{proof}[The proof of Proposition~\ref{takis3.1}] We claim that,  for  any $X_0, Y_0 \in L^2_{ac}$,  
\be\label{lk:jnezsrd:fg2}
\begin{split}
& \E\Big[ U(Y_0,\mathcal L(X_0)) -\beta M[U](Y_0, \mathcal L(X_0))- F( Y_0, \mathcal L( X_0))\\
&\qquad + H(  D_xU( Y_0, \mathcal L( X_0)),Y_0)\Big]\\ 
&-  \E \left[ D_mU( Y_0, \mathcal L( X_0),X_0)\cdot D_pH( D_xU( X_0, \mathcal L( X_0)),X_0)\right]\\
 &\leq \E\Big[ U(X_0,\mathcal L(X_0))-\beta M[U](X_0, \mathcal L(X_0)) - F( X_0, \mathcal L( X_0)) \\
& + H(  D_xU( X_0, \mathcal L( X_0)),X_0)\Big] \\
 &-  \E \left[ D_mU( X_0, \mathcal L( X_0),X_0)\cdot D_pH( D_xU( X_0, \mathcal L( X_0)),X_0)\right], 
\end{split}
\ee
where 
\begin{align*}
M[U](x,m) = & \Delta U(x,m) + \int_{\R^d} Tr(D^2_{ym}U(x,m,y))m(dy) +2\int_{\R^d} Tr(D^2_{xm} U(x,m,y))m(dy)\\
& +\int_{\R^{2d}} Tr(D^2_{mm}U(x,m,y,y'))m(dy)m(dy').
\end{align*}
Indeed, fix $\theta>0$. It follows from Ekeland-Lebourg \cite{EkLe76} or Stegall (see \cite{FF, St1, St2}) that, for any $\ep>0$, there exists $p_X,p_Y\in L^2$ such that $\|p^\ep_X\|_2+\|p^\ep_Y\|_2\leq \ep$ and the map 
\[
\begin{split}
(X,Y)\to \ & \widehat  U^\ep(X,Y)-  \E\left[ U(X, \mathcal L(X))-U(Y, \mathcal L(X))\right] \\[1.5mm]
&+\theta (\|X-X_0\|_2^2+\|Y-Y_0\|_2^2)-\E\left[p^\ep_X \cdot X+p^\ep_Y \cdot Y\right]
\end{split} \]
has a minimum at $(X_\ep,Y_\ep)$. 
\vs

Note that, as $\ep\to 0$,  $(X_\ep,Y_\ep) \to (X_0,Y_0)$ in $L^2\times L^2$ and, since $X_0\in L^2_{ac}$,  
$$
\ep\left\|\frac{d\mathcal L(X_\ep)}{d\lambda}\right\|_{\infty}\to 0.
$$ 
It follows from  Definition \ref{def.wealsol}, that 
\be\label{lk:jnezsrd:fg}
\begin{split}
0 \leq &\;  \hat U( X_\ep,  Y_\ep) -\beta \Lambda \left( \mathcal X \right)
 - \E\left[F( X_\ep, \mathcal L( X_\ep))-F( Y_\ep, \mathcal L( X_\ep))\right] \\
& +\E\left[ H(  D_xU( X_\ep, \mathcal L( X_\ep)),X_\ep)-H(-p_Y,Y_\ep)\right]\\
& - \E \left[ (D_xU( X_\ep, \mathcal L( X_\ep))-p_X)\cdot D_pH( D_xU( X_\ep, \mathcal L( X_\ep)),X_\ep)\right]\\
&  +C\ep\left(1+ \left\|\frac{d\mathcal L(X_\ep)}{d\lambda}\right\|_{\infty}\right), 
\end{split}
\ee
where, with  $m^\ep= \mathcal L(X_\ep)$, 
$$
p_X= D_xU(X_\ep, m_\ep)+ \int_{\R^d} (D_mU(X_\ep,m_\ep,y)-D_mU(Y_\ep,m_\ep,y))m_\ep(dy) -2\theta(X_\ep-X_0)+p^\ep_X,
$$
$$
p_Y= -D_xU(Y_\ep, m_\ep))-2\theta(Y_\ep-Y_0)+p^\ep_Y, 
$$
\begin{align*}
\mathcal X_{XX}& =  D^2_{xx}U(X_\ep, m_\ep)+ \int_{\R^{2d}} (D^2_{mm}U(X_\ep,m_\ep,y,,y')-D^2_{mm}U(Y_\ep,m_\ep, y,y'))m_\ep(dy)m_\ep(dy')\\
& +2 \int_{\R^d} D^2_{mx}U(X_\ep, m_\ep,y)m_\ep(dy) +\int_{\R^d} D^2_{ym}(U(X_\ep,m_\ep,y)-D^2_{ym}U(Y_\ep,m_\ep,y))m_\ep(dy)  -2\theta I, 
\end{align*}
$$
\mathcal X_{XY}=-\int_{\R^d} D^2_{mx}U(Y_\ep, m_\ep,y)m_\ep(dy) 
$$
and 
$$
\mathcal X_{YY}= -D^2_{xx}U(Y_\ep, m_\ep))-2\theta I. 
$$
In view of the definition of $\Lambda$, we have  
\begin{align*}
\Lambda(\mathcal X)& = \E\Bigl[ M[U](X_\ep, \mathcal L(X_\ep))- M[U](Y_\ep, \mathcal L(X_\ep))\Bigr]-4\theta d. 
\end{align*}
We can then pass to the limit in \eqref{lk:jnezsrd:fg} as $\ep\to 0$ and then as $\theta \to0$ to get \eqref{lk:jnezsrd:fg2}. 
\vs

Fix again $X_0\in L^2_{ac}$. Since \eqref{lk:jnezsrd:fg2} holds for any random variable $Y_0$, it holds in particular for any deterministic $Y_0\in \R^d$. Then, using the  assumption on $\mathcal L(X_0)$ and \eqref{lk:jnezsrd:fg2}, we find that
\begin{align*}
Y_0\to &\ U(Y_0,m_0) -\beta M[U](Y_0, m_0)- F( Y_0, m_0)+ H(  D_xU( Y_0, m_0),Y_0) \notag\\
& \qquad -  \int_{\R^d} D_mU( Y_0, m_0,y)\cdot D_pH( D_xU( y, m_0),y)m_0(dy)
\end{align*}
is constant, that is,  it is a map $g(m_0)$ which depends continuously on $m_0$ only. 
\vs 
It follows that  $U$ satisfies \eqref{ME2} at $m_0$ with a right-hand side given by $F(x,m_0)+g(m_0)$ instead of $F(x,m_0)$. Since  $L^2_{ac}$ is dense in $L^2$,  \eqref{ME2} holds everywhere (with right-hand side $F+g$ instead of $F$) . 

\end{proof}

\subsection*{The uniqueness of the weak solution} We now investigate the uniqueness of the weak solutions.

\begin{thm}\label{thm.unique} Assume \eqref{takis1002} and \eqref{takis1003}.
Then there exists at most one weak solution of  the master equation \eqref{ME2} up to an $m$-dependent constant. 
\end{thm}

\begin{proof} Let $U_1$ and $U_2$ be two weak solutions of  \eqref{ME2} and, for all  $(X,Y)\in L^2\times L^2$,  set
$$
\hat U_1(X,Y)= \E\left[ U_1(X, \mathcal L(X))-U_1(Y, \mathcal L(X))\right] \ \text{and} \ \hat U_2(X,Y)= \E\left[ U_2(Y, \mathcal L(Y))-U_2(X, \mathcal L(Y))\right]. 
$$
The goal is  to prove that 
\be\label{takis3.2}
\inf_{X,Y} \ [\hat U_1(X,Y)+\hat U_2(X,Y)]\geq 0,
\ee
which is equivalent to 
$$
\inf_{m,m'\in \Pw} \int_{\R^d} (U_1(x,m)-U_2(x,m'))(m(dx)-m'(dx))\; \geq \; 0.
$$
In view of Lemma~\ref{LionsLemma}  the last inequality  implies that $D_xU_1=D_xU_2$.
\vs
We begin the proof of \eqref{takis3.2} setting 
$$
M= \inf_{X,Y} \ [\hat U_1(X,Y)+\hat U_2(X,Y)],
$$
and considering, for $\ep>0$ and $\alpha>0$,  the map $\Phi_{\ep,\alpha}:L^2\times L^2\times L^2\times L^2 \to \R\cup \{+\infty\}$ defined by 
\begin{align*}
\Phi_{\ep,\alpha}(X,Y,X',Y')& = \E\left[ U_1(X, \mathcal L(X))-U_1(Y, \mathcal L(X))\right]  \\
& +  \E\left[ U_2(Y', \mathcal L(Y'))-U_2(X', \mathcal L(Y'))+\alpha(|X|^2+|Y'|^2)\right]\\
& +\frac{1}{2\alpha}(\|X-X'\|^2_2+\|Y-Y'\|^2_2)  + \ep \left( \left\|\frac{d\mathcal L(X)}{d\lambda}\right\|_{\infty}+\left\|\frac{d\mathcal L(Y')}{d\lambda}\right\|_\infty \right).
\end{align*}
Let 
$$
M_{\ep,\alpha}= \inf_{X,Y} \Phi_{\ep,\alpha}(X,Y,X',Y')
$$
and observe that, as $\ep,\alpha \to 0$, 
$$M_{\ep,\alpha}\to M.$$ 

Since  $\Phi_{\ep,\alpha}$ has a quadratic growth and is lower semicontinuous in $(L^2)^4$, we can find using  Stegall's Lemma, for any $\delta>0$,  $p_X, p_Y, p_{X'},p_{Y'}\in L^2$  such that 
\be\label{deltabound}
\|p_X\|_2+\| p_Y\|_2+\| p_{X'}\|_2+\|p_{Y'}\|_2\leq \delta
\ee
and the map 
$$
(X,Y,X',Y')\to \Phi_{\ep,\alpha}(X,Y,X',Y')-\E\left[ p_X\cdot  X+p_Y\cdot Y+p_{X'}\cdot X'+p_{Y'}\cdot Y'\right]
$$
has a minimum $M_{\ep,\alpha,\delta}$ at $(X_\delta, Y_\delta,X_\delta',Y_\delta')$. 
\vs

We note that, as $\delta\to 0$,  $M_{\ep,\alpha,\delta} \to M_{\ep,\alpha}$ 
and, for any $\kappa>0$, there exist  $\delta, \alpha >0$ and $\ep>0$ small enough so that 
\be\label{kahjzbksndfgc1}
\|X_\delta-X_\delta'\|^2_2+\|Y_\delta-Y_\delta'\|^2_2+\alpha(\|X_\delta\|^2_{2}+ \|Y_\delta'\|^2_{2})+ \frac{1}{2\alpha}(\|X_\delta-X_\delta'\|^2_{2}+\|Y_\delta-Y_\delta'\|^2_{2}) <\kappa
\ee
and 
\be\label{kahjzbksndfgc2}
 \ep \left( \left\|\frac{d\mathcal L(X_\delta)}{d\lambda}\right\|_\infty +\left\|\frac{d\mathcal L(Y_\delta')}{d\lambda}\right\|_\infty \right)<\kappa. 
 \ee
\vs

Following Lemma 4 of \cite{Lions89}, we can then find, for all $N\geq 1$, operators $\mathcal X_N$, $\mathcal Y_N$ such that $\mathcal X_N = P_N\mathcal X_NP_N$, $\mathcal Y_N = P_N\mathcal Y_NP_N$ (recall that $P_N$ and $Q_N$ are the projections onto $H^N$ and $H_N^\perp$ respectively; see section~1),  
\be\label{matrixineq}
-\frac{1}{\alpha}\left(\begin{array}{cc} I & -I\\-I& I\end{array}\right) \leq \left(\begin{array}{cc} \mathcal X_N & 0\\0& \mathcal Y_N\end{array}\right)  \leq 
\frac{2}{\alpha} \left(\begin{array}{cc} I & 0\\0& I\end{array}\right), 
\ee
$$
(\mathcal X_N+ \frac{1}{\alpha}  Q_N,- \frac{(X_\delta,Y_\delta)-(X_\delta',Y_\delta')}{\alpha}-2\alpha (X_\delta,0)+(p_X,p_Y)) \in \bar D^{2,-} \hat U_1^\ep(X_\delta,Y_\delta)
$$
and 
$$
(\mathcal Y_N+ \frac{2}{\alpha}  Q_N,   \frac{(X_\delta,Y_\delta)-(X_\delta',Y_\delta')}{\alpha}-2\alpha (0,Y_\delta')+(p_X',p_Y')) \in \bar D^{2,-} \hat U_2^\ep(X_\delta',Y_\delta'),
$$
where 
$$
\hat U_1^\ep(X,Y)=  \E\left[ U_1(X, \mathcal L(X))-U_1(Y, \mathcal L(X))\right]+ \ep \left\|\frac{d\mathcal L(X)}{d\lambda}\right\|_\infty 
$$
and
$$
\hat U_2^\ep(X',Y')= \E\left[ U_2(Y', \mathcal L(Y'))-U_2(X', \mathcal L(Y'))\right]+ \ep  \left\|\frac{d\mathcal L(Y')}{d\lambda}\right\|_{\infty}.
$$
It follows from the definition of weak solutions that 
\begin{align}\label{cond2BISU}
0 \leq &\;  \hat U_1( X_\delta,  Y_\delta) -\beta \Lambda \left( \mathcal X_N+\frac{1}{\alpha}Q_N \right)
 - \E\left[F( X_\delta, \mathcal L( X_\delta))-F( Y_\delta, \mathcal L( X_\delta))\right] \notag\\
& +\E\left[ H( D_xU_1( X_\delta, \mathcal L( X_\delta)), X_\delta)-H(\frac{Y_\delta-Y_\delta'}{\alpha}-p_Y,Y_\delta)\right]\\
& - \E \left[ (D_xU_1( X_\delta, \mathcal L( X_\delta))-(-\frac{X_\delta-X_\delta'}{\alpha}-2\alpha X_\delta+p_X))\cdot D_pH(  D_xU_1( X_\delta, \mathcal L( X_\delta)),X_\delta)\right]\notag\\
&  +C\ep\left(1+ \left\|\frac{d\mathcal L(X_\delta)}{d\lambda}\right\|_\infty\right),\notag
\end{align}
and,  for $\Sigma:L^2\times L^2\to L^2\times L^2$ defined by $\Sigma(X,Y)=(Y,X)$,
\begin{align}\label{cond2BISV}
0 \leq &\;  \hat U_2( X_\delta',  Y_\delta') -\beta \Lambda \left( \Sigma(\mathcal Y_N+\frac{1}{\alpha}Q_N) \right)
 - \E\left[F( Y_\delta', \mathcal L( Y_\delta'))-F( X_\delta', \mathcal L( Y_\delta'))\right] \notag\\
& +\E\left[ H( D_xU_2( Y_\delta', \mathcal L( Y_\delta')),Y_\delta')-H( - \frac{X_\delta-X_\delta'}{\alpha}- p_{X'},X_\delta')\right]\\
& - \E \left[ (D_xU_2( Y_\delta', \mathcal L( Y_\delta'))-(\frac{Y_\delta-Y_\delta'}{\alpha}-2\alpha Y_\delta'+p_Y'))\cdot D_pH(  D_xU_2( Y_\delta', \mathcal L( Y_\delta')),Y_\delta')\right]\notag\\
&  +C\ep\left(1+ \left\|\frac{d\mathcal L(Y_\delta')}{d\lambda}\right\|_\infty\right).\notag
\end{align}
We have already noticed that $\Lambda$ is linear with $\Lambda(Q_N)=0$. Moreover, in view of the definition of $\Lambda$, we also have $\Lambda\circ \Sigma =\Lambda$. Hence,  \eqref{matrixineq} implies 
$$
\Lambda \left( \mathcal X_N+\frac{1}{\alpha}Q_N \right)+\Lambda \left( \Sigma(\mathcal Y_N+\frac{1}{\alpha}Q_N )\right)\geq 0. 
$$
The  Lipschitz regularity and monotonicity  of $F$ also gives  
\begin{align*}
& \E\left[F( X_\delta, \mathcal L( X_\delta))-F( Y_\delta, \mathcal L( X_\delta))\right] 
  + \E\left[F( Y_\delta', \mathcal L( Y_\delta'))-F( X_\delta', \mathcal L( Y_\delta'))\right] \\
&\geq  -C(\|X_\delta-X_\delta'\|_{2}+\|Y_\delta-Y_\delta'\|_{2}). 
  \end{align*}
Using   the inequalities above in \eqref{cond2BISU} and \eqref{cond2BISV}  we find   
  \begin{align*}
0 \leq &\;  \hat U_1( X_\delta,  Y_\delta)+ \hat U_2( X_\delta',  Y_\delta') + C(\|X_\delta-X_\delta'\|_{2}+\|Y_\delta-Y_\delta'\|_{2})\\
& +\E\left[ H( D_xU_1( X_\delta, \mathcal L( X_\delta)), X_\delta)-H(\frac{Y_\delta-Y_\delta'}{\alpha}-p_Y,Y_\delta)\right]\\
& - \E \left[ (D_xU_1( X_\delta, \mathcal L( X_\delta))-(-\frac{X_\delta-X_\delta'}{\alpha}-2\alpha X_\delta+p_X))\cdot D_pH(  D_xU_1( X_\delta, \mathcal L( X_\delta)),X_\delta)\right]\notag\\
& +\E\left[ H( D_xU_2( Y_\delta', \mathcal L( Y_\delta')),Y_\delta')-H(  -\frac{X_\delta-X_\delta'}{\alpha}- p_{X'},X_\delta')\right]\\
& - \E \left[ (D_xU_2( Y_\delta', \mathcal L( Y_\delta'))-(\frac{Y_\delta-Y_\delta'}{\alpha}+2\alpha Y_\delta'+p_Y'))\cdot D_pH(  D_xU_2( Y_\delta', \mathcal L( Y_\delta')),Y_\delta')\right]\notag\\
&  +C\ep\left(1+ \left\|\frac{d\mathcal L(X_\delta)}{d\lambda}\right\|_{\infty}+ \left\|\frac{d\mathcal L(Y_\delta')}{d\lambda}\right\|_{\infty}\right).
\end{align*}
The Lipschitz regularity of $H$ (note that it is enough to assume that $H$ is only locally Lipschitz continuous)  and the fact that $D_xU_1$ and $D_xU_2$ are bounded, together with \eqref{deltabound}, 
allows to rewrite the last inequality as 
  \begin{align*}
0 \leq &\;  \hat U_1( X_\delta,  Y_\delta)+ \hat U_2( X_\delta',  Y_\delta') + C\alpha (\|X_\delta\|_{2}+\|Y'_\delta\|_{2})\\
& + C(\|X_\delta-X_\delta'\|_{2}+\|Y_\delta-Y_\delta'\|_{2})(1+\alpha^{-1}(\|X_\delta-X_\delta'\|_{2}+ \|Y_\delta-Y_\delta'\|_{2})+ \delta) \\
& - \E\left[ H( - \frac{X_\delta-X_\delta'}{\alpha}- p_{X'},X_\delta')- H( D_xU_1( X_\delta, \mathcal L( X_\delta)), X_\delta')\right]\\
& - \E \left[ -((-\frac{X_\delta-X_\delta'}{\alpha}+p_X)-D_xU_1( X_\delta, \mathcal L( X_\delta)))\cdot D_pH(  D_xU_1( X_\delta, \mathcal L( X_\delta)),X_\delta')\right]\notag\\
& -\E\left[ H(\frac{Y_\delta-Y_\delta'}{\alpha}-p_Y,Y_\delta)- H( D_xU_2( Y_\delta', \mathcal L( Y_\delta')),Y_\delta)\right]\\
& - \E \left[ -((\frac{Y_\delta-Y_\delta'}{\alpha}+p_Y')- D_xU_2( Y_\delta', \mathcal L( Y_\delta')))\cdot D_pH(  D_xU_2( Y_\delta', \mathcal L( Y_\delta')),Y_\delta)\right]\notag\\
&  +C\ep\left( 1+\left\|\frac{d\mathcal L(X_\delta)}{d\lambda}\right\|_{\infty}+ \left\|\frac{d\mathcal L(Y_\delta')}{d\lambda}\right\|_{\infty}\right).
\end{align*}
Since $H$ is convex in the first variable, we find, again due to \eqref{deltabound},  that 
  \begin{align*}
0 \leq &\;  \hat U_1( X_\delta,  Y_\delta)+ \hat U_2( X_\delta',  Y_\delta') + C\alpha (1+\|X_\delta\|_{2}^2+\|Y'_\delta\|_{2}^2) \\
& + C\Bigl(\|X_\delta-X_\delta'\|_{2}+\|Y_\delta-Y_\delta'\|_{2}+\delta\Bigr)\Bigl(1+\alpha^{-1}(\|X_\delta-X_\delta'\|_{2}+ \|Y_\delta-Y_\delta'\|_{2})+ \delta\Bigr) \\
&  +C\ep\left(1+ \left\|\frac{d\mathcal L(X_\delta)}{d\lambda}\right\|_{\infty}+ \left\|\frac{d\mathcal L(Y_\delta')}{d\lambda}\right\|_{\infty}\right).
\end{align*}
Recalling that, for any $\kappa>0$, we can find $\delta$, $\alpha$ and $\ep$ so small that \eqref{kahjzbksndfgc1} and \eqref{kahjzbksndfgc2} hold, we find  that 
$$
0\leq M_{\alpha, \ep,\delta} + C\kappa, 
$$
which in turn implies that $M\geq 0$ as $\delta, \ep, \alpha\to 0$. 

\end{proof}

\subsection*{The existence of weak solutions} The result is stated and proved next.

\begin{thm}\label{thm.existence} Assume \eqref{HH}, \eqref{FG}, \eqref{takis22} and \eqref{SM2}. Then there exists a weak solution of the master equation \eqref{ME3}. 
\end{thm}

\begin{proof} The construction of a weak solution relies on the stochastic MFG system studied in \cite{CaSo20}. For this,  we fix a Brownian motion $(W_t)_{t\geq 0}$ defined on a probability space $(\Omega', \mathcal F', \P')$ which is  independent of the space $(\Omega, \mathcal F, \P)$ on which we develop the notion of weak solution. Abusing the notation  we still denote by $\E$ the expectation with respect to the product measure $\P\otimes \P'$. 
\vs
For $t_0\geq0$, let  $(\tilde u,\tilde m, \tilde M)$ be the solution of the  system 
\be\label{stoMFG_Intro}
\left\{\begin{array}{l}
\ds d_t\tilde  u_t = \left[\tilde u_t(x)+ \tilde H_{t_0,t} (D\tilde u_{t}(x),x)-\tilde F_{t_0,t}(x,\tilde m_t)\right]dt +d\tilde M_t \ \    {\rm in} \  \ \R^d\times (t_0,+\infty), \\[2mm] 
\ds \partial_t\tilde m_t =  {\rm div}(\tilde m_tD_p\tilde H_{t_0,t}(D\tilde u_t(x),x)) \ \  {\rm in} \ \ \R^d\times (t_0,+\infty), \\[2mm] 
 \ds \tilde m_{t_0}=m_0 \ \ \text{in} \ \ \R^d, 
\end{array}\right.
\ee
with  
\be\label{takis0}
\begin{split}
&\tilde H_{t_0,t}(p,x)= H(p, x+\sqrt{2\beta} (W_t-W_{t_0})), \\
& \tilde F_{t_0,t}(x,m)= F(x+\sqrt{2\beta}(W_t-W_{t_0}),(id+\sqrt{2\beta}(W_t-W_{t_0}))\sharp m).
\end{split}
\ee
We recall from \cite{CaSo20} that $(\tilde u,\tilde m, \tilde M)$  is an adapted process such that, for a.e. $x\in \R^d$, $(\tilde M_t(x))$ is a martingale, $\tilde u$ solves the first equation a.s. and a.e. and $\tilde m$ solves the second equation a.s. in the sense of distributions. 
\vs

It was shown in Lemma 3.6,  Lemma 3.4 and the proof of Theorem 3.3  all  in \cite{CaSo20} that, if  \eqref{HH} and \eqref{takis22} hold, then \eqref{stoMFG_Intro} has a solution such that, for some  $C_0>0$ which depends only on $H$ and $F$ and all $t\in (t_0,\infty)$ and $z \in \R^d$,  
\be\label{takis100}
\begin{split}
& \|\tilde u_t\|_\infty  + \|D \tilde u_t\|_\infty + \|\tilde M_t\|_\infty + D^2\tilde u_t z\cdot z\leq C,  \ \text{and}\\[1.5mm]
& \text{for a.e. $x\in \R^d$, the process \  $(\tilde M(x))_{t\geq t_0})$ is a continuous martingale}, 
\end{split} 
\ee
and, if $m_0\in L^\infty$ and $M_2(m_0)<+\infty$, then  
\be\label{estierdfn}
\|\tilde m_t\|_\infty\leq  \|m_0\|_\infty e^{C_0(t-t_0)} \ \text{and}  \  M_2(\tilde m_t)\leq M_2(m_0)e^{C_0(t-t_0)} \qquad a.s.
\ee

Note that, since $\tilde u$ is adapted to the filtration generated by $(W_t-W_{t_0})_{t\geq t_0}$,  $\tilde u_{t_0}(x)$ is deterministic and   independent of $t_0$. 
\vs
Let 
 \[U(x,m_0)=\tilde u_{t_0}(x).\]
 
It follows from \eqref{takis100} that $U$ is Lipschitz continuous and semiconcave with respect to $x$ uniformly in $m$. 
\vs
Moreover, $U$ admits a continuous extension on $\R^d\times \mathcal P_2(\R^d)$. This is a consequence of the fact that there exists $C_0>$ such that,  for any  $m_0,m_0'\in \mathcal P_2(\R^d)$ which are absolutely continuous with a bounded density, 
\be\label{takis3.4}
|U(x,m_0)-U(x,m'_0)|\leq C_0{\bf d}_1(m_0,m_0'))^{1/(d+2)}.
\ee
This last estimate follows from Lemma~\ref{takis3.5} which is stated and proved after the end of the ongoing proof. 
\vs
The  aim is to show that $U$, which, in view of \eqref{takis100},  is  bounded, continuous in $(x,m)$ and Lipschitz continuous and semiconcave in $x$ uniformly  in $m$, is a weak solution to \eqref{ME2}.
\vs

The relation between the MFG system \eqref{stoMFG_Intro} and the master equation \eqref{ME2}  is explained from the fact that, for any $(h,x)\in (0,+\infty)\times \R^d$ and a.s., 
\be\label{rel.DynBIS}
U(x+\sqrt{2\beta}(W_{t_0+h}-W_{t_0}), (id+\sqrt{2\beta}(W_{t_0+h}-W_{t_0})) \sharp \tilde m_{t_0+h}) = \tilde u_{t_0+h}(x).
\ee
This  is the subject of Lemma~\ref{takis3.6} which is  stated and proved after the end of the ongoing proof. 
\vs

Following the discussion about the connection between subdifferentials and subjets in section~1 and Definition~\ref{def.wealsol}, we fix a  $C^2$-test function $\Phi: L^2\times L^2\to \R$ and assume that the map 
\be\label{cond.test0}
\begin{split}
\ds & (X,Y) \to  \E\left[ U(X, \mathcal L(X))-U(Y, \mathcal L(X))\right]
-\Phi(X,Y)
+\ep  \left\|\frac{d\mathcal L(X)}{d\lambda}\right\|_{\infty}\\[1.5mm]
& \text{achieves a minimum $I$ at $(\bar X, \bar Y)$.}
\end{split}
\ee 

Note that, without loss of generality, we may  assume that this minimum is strict and that $-\Phi$ has a quadratic growth. 
\vs

We claim that 
\be\label{takis101}
\begin{split}
0 \leq &\;  \hat U( \bar X, \bar Y) -\beta \Lambda \left(D^2_{(\bar X,\bar Y)}\Phi(\bar  X,\bar Y)\right)
 - \E\left[F( \bar X, \mathcal L( \bar X))-F( \bar Y, \mathcal L( \bar X))\right] \\
& +\E\left[ H( D_xU( \bar X, \mathcal L( \bar X)),\bar X)-H(- D_Y\Phi(\bar X, \bar Y), \bar Y)\right]\\
& - \E \left[ (D_xU( \bar X, \mathcal L( \bar X))-D_X\Phi(\bar X,\bar Y))\cdot D_pH(D_xU( \bar X, \mathcal L( \bar X)), \bar X)\right]\\
&  +C\ep\left(1+ \left\|\frac{d\mathcal L(\bar X)}{d\lambda}\right\|_{\infty}\right), 
\end{split}
\ee
which is the condition needed for $U$ to be a weak solution. 
 \vs

In order to handle terms of the form $H(D_xU(Y,\mathcal L(X)), Y)$, we need to regularize $U$ with respect to the space variable. For this,  we  fix a smooth, nonnegative kernel with compact support $\xi$ and, for $\eta\in(0,\ep)$ small, we consider the mollifier $\xi_\eta(x)= \eta^{-d}\xi(x/\eta)$. 
\vs
The uniform in $m$ Lipschitz continuity of  $U$ with respect to $x$ and the Lipschitz  continuity of $\tilde u_t$ yield that, for a uniformly small $\eta$,  
\be\label{takis102}
\|\xi_\eta\ast_x U(\cdot, m)-U(\cdot, m)\|\leq C\eta \ \ \text{and} \ \ \|\xi_\eta\ast_x \tilde u_t -\tilde u_t\|\leq C\eta.
\ee
\vs

It follows from   Stegall's Lemma, that, for all $\eta>0$ small, there exist  $p_X, p_Y\in L^2$ such that 
$$
\|p_X\|_{2}+\|p_Y\|_{2} \leq \eta
$$
and the map 
\be\label{cond.test}
\begin{split}
(X,Y)\to &  \E\left[ U(X, \mathcal L(X))-\xi_\eta\ast U(\cdot, \mathcal L(X))(Y)\right]
-\Phi(X,Y)-\E\left[ p_X\cdot X+p_Y\cdot Y\right]\\
& +\ep  \left\|\frac{d\mathcal L(X)}{d\lambda}\right\|_{\infty}\\
& \text{achieves a minimum $I_\eta$ at some point $(\bar X_\eta, \bar Y_\eta)\in L^2_{ac}\times L^2$.}
\end{split}
\ee
%
The main step of the ongoing  proof  is to show that 
\be\label{cond2TER}
\begin{split}
0 \leq &\;  \hat U( \bar X_\eta,  \bar Y_\eta) -\beta \Lambda \left( D^2_{(X,Y)}\Phi(\bar X_\eta,\bar Y_\eta) \right)
 - \E\left[F( \bar X_\eta, \mathcal L( \bar X_\eta))-F( \bar Y_\eta, \mathcal L( \bar X_\eta))\right]\\
& +\E\left[ H(  D_xU( \bar X_\eta, \mathcal L( \bar X_\eta)),\bar X_\eta)-H(-D_Y\Phi(\bar X_\eta, \bar Y_\eta)-p_Y,\bar Y_\eta)\right]\\
& - \E \left[ (D_xU( \bar X_\eta, \mathcal L( \bar X_\eta))-D_X\Phi(\bar X_\eta,\bar Y_\eta)-p_X)\cdot D_pH( D_xU( \bar X_\eta, \mathcal L( \bar X_\eta)),\bar X_\eta)\right]\\
&  +C\ep\left(1+ \left\|\frac{d\mathcal L(\bar X_\eta)}{d\lambda}\right\|_{\infty}\right)+C\eta
\end{split}
\ee
and 
\be\label{cond2TER+}
|D_Y\Phi(\bar X_\eta, \bar Y_\eta)+p_Y|\leq C,\;  {\rm a.s.}
\ee
\vs

We continue with the proof of  \eqref{takis101} and establish   \eqref{cond2TER} and \eqref{cond2TER+} later.
\vs 
For the remainder  of the argument all the limits are taken as $\eta\to 0$. Hence we will not be repeating this fact.
\vs

It is clear that $I_\eta \to I$. Moreover, the fact that  the minimum in \eqref{cond.test0} is strict yields  that \[(\bar X_\eta, \bar Y_\eta) \to (\bar X,\bar Y) \ \ \text{in $L^2\times L^2$,}\] 
\vs
%
and, thus,  
\[\left\|\dfrac{d\mathcal L(\bar X_\eta)}{d\lambda}\right\|_{\infty} \to \left\|\dfrac{d\mathcal L(\bar X)}{d\lambda}\right\|_{\infty}.\]

In addition, since  $ \bar X_\eta \to \bar X$ in $L^2$,  it follows that the density of $m^\eta=\mathcal L(\bar X_\eta)$, which is uniformly bounded, converges weakly-$\star$ to the density of $m=\mathcal L(\bar X)$. 
\vs

The uniform continuity of $U$  in both variables and the uniform semiconcavity in $x$ allows to pass to the limit   in the terms 
\vskip-.1in
\[
\E\left[ H(  D_xU( \bar X_\eta, \mathcal L( \bar X_\eta)),\bar X_\eta)\right]
= \int_{\R^d} H(D_xU(x,m^\eta),x)m^\eta(x)dx 
\]
and 
\begin{align*}
&\E \left[ (D_xU( \bar X_\eta, \mathcal L( \bar X_\eta)) \cdot D_pH( D_xU( \bar X_\eta, \mathcal L( \bar X_\eta)),\bar X_\eta)\right]\\
& \qquad=
\int_{\R^d}(D_xU(x, m^\eta)\cdot D_pH( D_xU( x, m^\eta),x)) m^\eta(x)dx .
\end{align*}
Similarly,  since 
\[\begin{split}
D_X\Phi(\bar X_\eta,\bar Y_\eta)+p_X \to D_X\Phi(\bar X,\bar Y) \ \ \text{in $L^2$ \ 
and}\\[1.5mm]
\text{$D_pH( D_xU( \bar X_\eta, \mathcal L( \bar X_\eta)),\bar X_\eta) \to D_pH( D_xU( \bar X, \mathcal L( \bar X)),\bar X)$ a.s.},
\end{split}
\]

 it is possible to pass in the  limit  in 
$$
\E\left[(-D_X\Phi(\bar X_\eta,\bar Y_\eta)-p_X)\cdot D_pH( D_xU( \bar X_\eta, \mathcal L( \bar X_\eta)),\bar X_\eta)\right].
$$
Finally, since $D_Y\Phi(\bar X_\eta, \bar Y_\eta)+p_Y \to D_Y\Phi(\bar X, \bar Y)$ in $L^2$, in view of  \eqref{cond2TER+}, we can also pass to the limit  in 
$\ds \E\left[ H(-D_Y\Phi(\bar X_\eta, \bar Y_\eta)-p_Y,\bar Y_\eta)\right].$
\vs
In conclusion, one can pass to the limit in the whole expression \eqref{cond2TER} and obtain \eqref{takis101}.\\

We now return  to the proofs of \eqref{cond2TER} and \eqref{cond2TER+}. To  simplify  the notation, we  write $\bar X$, $\bar Y$ and $\overline \Phi$ for $\bar X_\eta$, $\bar Y_\eta$ and $\Phi+ \E[p_X\cdot  X +p_Y\cdot  Y]$ respectively, and  assume that the map 
\be\label{cond.test2}
\begin{split}
\hskip-.75in  (X,Y)\to  \E\left[ U(X, \mathcal L(X))-\xi_\eta\ast U(\cdot, \mathcal L(X))(Y)\right] 
-\overline \Phi(X,Y)\\[2mm]
 +\ep \left\|\frac{d\mathcal L(X)}{d\lambda}\right\|_{\infty} \ \ \text{ achieves  a minimum at $(\bar X, \bar Y)$}
\end{split}
\ee
which,  without loss of generality, we assume that is $0$. 

\vs 
%
%

Let $(\tilde u, \tilde m)$ be the solution of \eqref{stoMFG_Intro} with initial condition $\bar m_{0}= \mathcal L(\bar X)$. 
\vs
In order to use \eqref{cond.test2}, we now need to lift the (random) flow $(\tilde m_t)_{t\geq t_0}$ to $L^2$. 
\vs
The natural thing  to do is to find a solution $(\phi^x)_{t\geq t_0}$ of  the ode (with random coefficients)
\be\label{defX}
 \frac{d}{dt} \phi^x_t= -D_p\tilde H_{t_0,t} (D\tilde u_t(\phi^x_t),\phi^x_t) \ \text{in} \ (t_0,\infty), \ \ \  \  \phi^x_{t_0}=x, 
\ee
which is adapted to the filtration generated by $(W_t-W_{t_0})_{t\geq t_0}.$ Then one would expect that $\tilde m_t(x)= \phi^{x}\sharp \bar m_0$, so that $X_t= \phi^{\bar X}_t$ would have the property that  $\tilde m_t= \mathcal L(X_t|W)$.
\vs
Unfortunately,  the existence of such a flow is not known in general without adding some randomness to the flow or some extra structure condition on the data; see the discussion in Section 2.6 of \cite{CaSo20}. 
\vs
To overcome this issue, we proceed by approximation. It follows from  Lemma 3.6, Lemma 3.8 and the proof of Theorem 3.8 in \cite{CaSo20} that there exists a sequence $(\tilde u^N, \tilde m^N, \tilde M^N)_{N\geq 1}$ such that, for any $T>t_0$ and $R>0$, 
\be\label{kjhzesbrdf}
\underset{N\to\infty}\lim \sup_{t\in [t_0,T]}  \E\left[\|\tilde u_t-\tilde u^N_t\|_{L^\infty (B_R)}^{d+1}\right]=0,  \ \  \ \  
\|\tilde m^N_t\|_\infty\leq  \|m_0\|_\infty e^{C_0(t-t_0)} 
\ee
and, a.s., 
$$
\underset{N\to\infty}\lim\tilde m^N= \tilde m\; {\rm in}\; C^0([0,T], \mathcal P_2(\R^d)) \  \text{and in }  \ L^\infty-{\rm weak-}\star. 
$$
Then, following \cite{CaHa},  we can solve, for a.e. $x\in \R^d$,  the ode 
\be\label{defXN}
 \frac{d}{dt} \phi^{N,x}_t= -D_p\tilde H^N_{t_0,t} (D\tilde u^N_t(\phi^{N,x}_t),\phi^{N,x}_t), \qquad \phi^{N,x}_{t_0}=x
\ee
in a unique way and, as shown in  \cite{CaHa}, we have $\tilde m^N_t(x)= \phi^{N,x}_t\sharp \bar m_0$. 
\vs
We  set $X^N_t= \phi^{N, \bar X}_t$ and remark that by definition  $\tilde m^N_t= \mathcal L(X^N_t|W)$; note that $X^N_t=X^N_t(\omega,\omega')$ where $(\omega,\omega')\in \Omega\times \Omega'$.  
\vs
If $\Psi:L^2\to \R$ is continuous, we denote by 
$\Psi(X^N_t)$ the random variable $\omega'\to \Psi(X^N_t(\cdot, \omega'))$ on $\Omega'$, and observe that 
$$
 \Psi(X^N_t+\sqrt{2\beta}(W_{t_0+h}-W_{t_0}))= \Psi(X^N_t(\cdot, \omega')+z)_{z=\sqrt{2\beta}(W_{t_0+h}-W_{t_0})(\omega')}.
$$
Set $\tilde u^\eta_{t}(x)=( \xi_\eta\ast \tilde u_t)(x)$ and,  for all  $(X,Y)\in L^2\times L^2$, 
$$
\hat U^\eta(X,Y)= \E\left[ U(X, \mathcal L(X))-\xi_\eta\ast U(\cdot, \mathcal L(X))(Y)\right].
$$

To complete the ongoing proof we need two additional results which we state below as separate lemmata and present their proof later. 
\begin{lem}\label{lem.laesndfm} Fix $h>0$. For $N$ large enough, depending on $h$,  we have 
\begin{align*}
& \E\Bigl[ \hat U^\eta(X^N_{t_0+h}+\sqrt{2\beta}(W_{t_0+h}-W_{t_0}), \bar Y+\sqrt{2\beta}(W_{t_0+h}-W_{t_0}))\Bigr]-\hat U^\eta(\bar X, \bar Y)\\
& \leq  h \E\left[ (\tilde u_{t_0}(\bar X)+ H(D\tilde u_{t_0}(\bar X),\bar X)-D_p H(D\tilde u_{t_0}(\bar X),\bar X)\cdot D \tilde u_{t_0}(\bar X)- F(\bar X, \bar m_0) \right]\\
& \qquad -h \E\left[ \tilde u_{t_0}(\bar Y)+  H(- D_Y\Phi(\bar X,\bar Y),\bar Y)-  F(\bar Y, \bar m_0) \right]+C\eta h+o(h),
\end{align*}
where $\bar m_0=\mathcal L(\bar X)$. In addition, \eqref{cond2TER+} holds. 
\end{lem}

\begin{lem}\label{lem.laesndfm2} For $N$ large enough depending on $h$, 
\begin{align*}
&   \E\Bigl[  \Phi(X^N_{t_0+h}+\sqrt{2\beta}(W_{t_0+h}-W_{t_0}), \bar Y+\sqrt{2\beta}(W_{t_0+h}-W_{t_0}))\Bigr]\\
&\geq   \Phi(\bar X,\bar Y)+ h \E\Bigl[  - D_p H (D\tilde u_{t_0}(\bar X),\bar X)\cdot D_X\Phi(\bar X,\bar Y)\Bigr]+\beta h \Lambda (D^2_{(X,Y)}\Phi(\bar X,\bar Y))+o(h).
 \end{align*}
\end{lem}

To prove  \eqref{cond2TER} we recall that  the minimum in \eqref{cond.test2} is assumed to be $0$, and 
we find, using \eqref{estierdfn}, \eqref{kjhzesbrdf} and Lemma~\ref{lem.laesndfm},  that, for $N$ large enough, 
\begin{align*}
&  \E\Bigl[ \Phi(X^N_{t_0+h}+\sqrt{2\beta}(W_{t_0+h}-W_{t_0}), \bar Y+\sqrt{2\beta}(W_{t_0+h}-W_{t_0}))\Bigr]\\
& \leq \E\Bigl[ \hat U^\eta(X^N_{t_0+h}+\sqrt{2\beta}(W_{t_0+h}-W_{t_0}), \bar Y+\sqrt{2\beta}(W_{t_0+h}-W_{t_0}))\Bigr]\\
& \quad + \ep\E'\Bigl[ \left\|\frac{d\mathcal L(X_{t_0+h}^N-\sqrt{2\beta}(W_{t_0+h}-W_{t_0})|W)}{d\lambda}\right\|_{\infty} \Bigr] \\ 
& \leq \hat U^\eta(\bar X, \bar Y) +  h \E\left[ (\tilde u_{t_0}(\bar X)+ H(D\tilde u_{t_0}(\bar X),\bar X)-D_p H(D\tilde u_{t_0}(\bar X),\bar X)\cdot D u_{t_0}(\bar X)- F(\bar X, \bar m_0) \right]\\
& \qquad -h \E\left[ \tilde u_{t_0}(\bar Y)+  H(- D_Y\Phi(\bar X,\bar Y),\bar Y)-  F(\bar Y, \bar m_0) \right]\\
& \qquad + \ep (1+Ch) \|\bar m_0\|_\infty+ C\eta h+o(h). 
\end{align*}
We have also seen from  Lemma \ref{lem.laesndfm2} that, for $N$ large enough, 
\begin{align*}
&   \E\Bigl[  \Phi(X^N_{t_0+h}+\sqrt{2\beta}(W_{t_0+h}-W_{t_0}), \bar Y+\sqrt{2\beta}(W_{t_0+h}-W_{t_0}))\Bigr]\\
&\geq   \Phi(\bar X,\bar Y)+ h \E\Bigl[  - D_p H (D\tilde u_{t_0}(\bar X),\bar X)\cdot D_X\Phi(\bar X,\bar Y)\Bigr]+\beta h \Lambda (D^2_{(X,Y)}\Phi(\bar X,\bar Y))+o(h).
 \end{align*}
 Combining the last two inequalities and using that  $\hat U^\eta(\bar X,\bar Y)= \Phi(\bar X,\bar Y)-\ep \|\bar m_0\|_\infty$ and $D_xU(x,\bar m_0)=D\tilde u_{t_0}(x)$, we get, letting $h\to 0$, 
 \begin{align*}
&   \E\Bigl[  - D_p H (D_xU(\bar X, \mathcal L(\bar X)),\bar X)\cdot D_X\Phi(\bar X,\bar Y)\Bigr]+\beta \Lambda (D^2_{(X,Y)}\Phi(\bar X,\bar Y)) \\
& \leq \E\left[U(\bar X,\mathcal L(\bar X))+ H(D_xU(\bar X,\mathcal L(\bar X)),\bar X)\right]\\
& \qquad -\E\left[ D_pH(D_xU(\bar X,\mathcal L(\bar X)),\bar X)\cdot D_xU(\bar X,\mathcal L(\bar X))-F(\bar X,\mathcal L(\bar X)) \right]\\ 
& \qquad -  \E\left[U(\bar Y,\mathcal L(\bar X))+ H(- D_Y\Phi(\bar X,\bar Y),\bar Y)- F(\bar Y,\mathcal L(\bar X)) \right]\\ 
& \qquad + C \ep  (1+\|\bar m_0\|_\infty)+C\eta,
\end{align*}
and, after some rearranging, 
\begin{align*}
0 \leq &\;  \hat U( \bar X, \bar Y) -\beta \Lambda \left(D^2_{(\bar X,\bar Y)}\Phi(\bar  X,\bar Y)\right)
 - \E\left[F( \bar X, \mathcal L( \bar X))-F( \bar Y, \mathcal L( \bar X))\right] \notag\\
& +\E\left[ H( D_xU( \bar X, \mathcal L( \bar X)),\bar X)-H(- D_Y\Phi(\bar X, \bar Y), \bar Y)\right]\\
& - \E \left[ (D_xU( \bar X, \mathcal L( \bar X))-D_X\Phi(\bar X,\bar Y))\cdot D_pH( \bar X, D_xU( \bar X, \mathcal L( \bar X)))\right]\notag\\
&  +C\ep\left(1+ \left\|\frac{d\mathcal L(\bar X)}{d\lambda}\right\|_{\infty}\right)+C\eta, \notag
\end{align*}
which is \eqref{cond2TER}. 

\end{proof}

We continue with the statements and proofs of the technical facts used in the previous proof.

\begin{lem}\label{takis3.5} Assume \eqref{HH} and \eqref{takis22}. Then there exists $C_0>0$ such that, for all $m_0,m_0'\in \mathcal P_2(\R^d)$ which are absolutely continuous with  bounded density, \eqref{takis3.4} holds. 
\end{lem}

\begin{proof} We only present a sketch, since the complete proof can be concluded  by standard arguments. 
\vs
Let $(\tilde u,\tilde m,\tilde M)$ and $(\tilde u',\tilde m',\tilde M')$ be the solutions of  \eqref{stoMFG_Intro} with initial condition $m_0$ and $m_0'$ respectively. Following the proof of Lemma 3.7 in \cite{CaSo20}, we have 
\begin{align*}
&\E\left[ \int_{t_0}^{+\infty}\int_{\R^d}e^{-t} \left(\tilde F_{t_0,t}(x, \tilde m_t)-\tilde F_{t_0,t}(x, \tilde m_t')\right)(\tilde m_t(x)-\tilde m_t'(x))dxdt\right] \\
& \qquad \leq -\E\left[
\int_{\R^d} (\tilde u_{t_0}(x)-\tilde u'_{t_0}(x))(m_0(x)-m_0'(x))dx\right]. 
\end{align*}
The strong monotonicity  of  $F$ and its Lipschitz regularity in $x$ uniformly in $m$ on the one hand, and the uniform Lipschitz regularity in $x$ of $\tilde u$ and $\tilde u'$ on the other hand, yield by an interpolation inequality
\[
\alpha_F \E\left[ \int_0^{+\infty}e^{-t} \left\|\tilde F_{t_0,t}(\cdot, \tilde m_t)-\tilde F_{t_0,t}(\cdot, \tilde m_t')\right\|_\infty^{d+2} dt\right]  \leq C{\bf d_1}(m_0,m_0'). 
\]
Then  the optimal control representation of the solution (Proposition 2.7 of \cite{CaSo20}) and H\"{o}lder's inequality give
\[
|\tilde u_{t_0}(x)-\tilde u'_{t_0}(x)| \leq \E\left[ \int_0^{+\infty}e^{-t} \left\|\tilde F_{t_0,t}(\cdot, \tilde m_t)-\tilde F_{t_0,t}(\cdot, \tilde m_t')\right\|_\infty dt\right] \leq C{\bf d_1}^{1/(d+2)}(m_0,m_0').
\]

Using the definition of $U$, we may now  conclude. 

\end{proof}

\begin{lem} \label{takis3.6} For any $(h,x)\in (0,+\infty)\times \R^d$ and a.s. 
\be\label{rel.Dyn}
U(x+\sqrt{2\beta}(W_{t_0+h}-W_{t_0}), (id+\sqrt{2\beta}(W_{t_0+h}-W_{t_0})) \sharp \tilde m_{t_0+h}) = \tilde u_{t_0+h}(x).
\ee
\end{lem}

\begin{proof} Let $(\tilde u^h, \tilde m^h, \tilde M^h)$ be defined, for $t\geq t_0+h$ and $x\in \R^d$, by 
$$
\tilde u^h_t(x)= \tilde u_t(x-\sqrt{2\beta}(W_{t_0+h}-W_{t_0})), \; \tilde m^h_t= (id+{2\beta}(W_{t_0+h}-W_{t_0})) \sharp \tilde m_{t}
$$
and
$$
\tilde M^h_t(x)= \tilde M_t(x-\sqrt{2\beta}(W_{t_0+h}-W_{t_0}))-\tilde M_{t_0+h}(x-\sqrt{2\beta}(W_{t_0+h}-W_{t_0})). 
$$

Recalling \eqref{takis0},   we find, for a.e. $x\in \R^d$, any $t\geq t_0+h$ and a.s., 
\begin{align*}
& \tilde  u^h_t(x)-\tilde  u^h_{t_0+h}(x)  = \tilde u_t(x-\sqrt{2\beta}(W_{t_0+h}-W_{t_0}))-\tilde u_{t_0+h}(x-\sqrt{2\beta}(W_{t_0+h}-W_{t_0}))\\
& = \int_{t_0+h}^t \Big[\tilde u^h_s(x)+ \tilde H_{t_0,s} (D\tilde u^h_{s}(x),x-\sqrt{2\beta}(W_{t_0+h}-W_{t_0}))\\
& \qquad -\tilde F_{t_0,s}(x-\sqrt{2\beta}(W_{t_0+h}-W_{t_0}),\tilde m_s)\Big]ds 
   +\tilde M^h_{t}(x)-\tilde M^h_{t_0+h}(x)  \\
& = \int_{t_0+h}^t \left[\tilde u^h_s(x)+ \tilde H_{t_0+h,s} (D\tilde u^h_{s}(x),x)-\tilde F_{t_0+h,s}(x,\tilde m^h_s)\right]ds +\tilde M^h_{t}(x)-\tilde M^h_{t_0+h}(x),
\end{align*}
since, for $s\geq t_0+h$, 
\begin{align*}
\tilde F_{t_0+h,s}(x,\tilde m^h_s)& = F(x+\sqrt{2\beta}(W_s-W_{t_0+h}),(id+\sqrt{2\beta}(W_s-W_{t_0+h}))\sharp \tilde m^h_s) \\ 
& = \tilde F_{t_0, s}(x -\sqrt{2\beta}(W_{t_0+h}-W_{t_0}), (id-\sqrt{2\beta}(W_{t_0+h}-W_{t_0}))\sharp \tilde m^h_s )\\ 
& = \tilde F_{t_0, s}(x -\sqrt{2\beta}(W_{t_0+h}-W_{t_0}),  \tilde m_s ).
\end{align*}
A similar argument shows  that $\tilde m^h$ is a.s. a weak solution of 
$$
\ds \partial_t\tilde m^h_t =  {\rm div}(\tilde m^h_tD_p\tilde H_{t_0+h,t}(D\tilde u^h_t(x),x)) \ \  {\rm in} \ \ \R^d\times (t_0,+\infty). 
$$
This proves that $(\tilde  u^h , \tilde m^h, \tilde M^h)$ solves \eqref{stoMFG_Intro} on the time interval $(t_0+h, +\infty)$ and  with the initial condition $\tilde m^h_{t_0+h}$. Therefore $U(x,\tilde m^h_{t_0+h})= \tilde u^h_{t_0+h}(x)$ a.s., which implies the result. 

\end{proof}

\begin{proof}[The proof of Lemma~\ref{lem.laesndfm}.] 
The definition of $\hat U^\eta$ gives 
\begin{align*}
& \hat U^\eta(X^N_{t_0+h}+\sqrt{2\beta}(W_{t_0+h}-W_{t_0}), \bar Y+\sqrt{2\beta}(W_{t_0+h}-W_{t_0}))\\
& \qquad = U(X^N_{t_0+h}+\sqrt{2\beta}(W_{t_0+h}-W_{t_0}), \mathcal L(X^N_{t_0+h}+\sqrt{2\beta}(W_{t_0+h}-W_{t_0})|W)) \\
& \qquad \qquad -
\xi_\eta\ast U(\cdot, \mathcal L(X^N_{t_0+h}+\sqrt{2\beta}(W_{t_0+h}-W_{t_0})|W))(\bar Y+\sqrt{2\beta}(W_{t_0+h}-W_{t_0})), 
\end{align*}
where 
\begin{align*}
& \mathcal L(X^N_{t_0+h}+\sqrt{2\beta}(W_{t_0+h}-W_{t_0})|W) = (id+ \sqrt{2\beta}(W_{t_0+h}-W_{t_0}))\sharp \mathcal L(X^N_{t_0+h}|W)\\
&\qquad = 
(id+ \sqrt{2\beta}(W_{t_0+h}-W_{t_0}))\sharp \tilde m^N_{t_0+h}.
\end{align*}
Since,  as $N\to +\infty$ a.s.,  $\tilde m^N_{t_0+h}$ converges weakly to $\tilde m_{t_0+h}$ a.s., 
it follows that, for $N$ large enough, 
\begin{align*}
& \E'\left[ \hat U^\eta(X^N_{t_0+h}+\sqrt{2\beta}(W_{t_0+h}-W_{t_0}), \bar Y+\sqrt{2\beta}(W_{t_0+h}-W_{t_0}))\right]\\
& \qquad \leq \E\Bigl[ U(X^N_{t_0+h}+\sqrt{2\beta}(W_{t_0+h}-W_{t_0}), (id+ \sqrt{2\beta}(W_{t_0+h}-W_{t_0}))\sharp \tilde m_{t_0+h}) \\
& \qquad \qquad -
\xi_\eta\ast U(\cdot, (id+ \sqrt{2\beta}(W_{t_0+h}-W_{t_0}))\sharp \tilde m_{t_0+h})(\bar Y+\sqrt{2\beta}(W_{t_0+h}-W_{t_0}))\Bigr] +h^2/2\\ 
& \qquad = \E\left[ \tilde u_{t_0+h}(X^N_{t_0+h})- \tilde u_{t_0+h}^\eta(\bar Y)\right] +h^2/2, 
\end{align*}
the last equality coming from \eqref{rel.Dyn} and the definition of $\tilde u^\eta$.
\vs

Since 
$$
\E\left[ \tilde u_{t_0+h}(X^N_{t_0+h})\right]= \E\left[ \int_{\R^d} \tilde u_{t_0+h}(x) \tilde m^N_{t_0+h}(x)dx\right] \to \E\left[ \int_{\R^d} \tilde u_{t_0+h}(x) \tilde m_{t_0+h}(x)dx\right]
$$
using the weak convergence of $\tilde m^N$ to $\tilde m$, we can find $N$ large enough such that 
\be\label{lerhzilrneskd1}
\begin{split}
& \E'\Bigl[ \hat U^\eta(X^N_{t_0+h}+\sqrt{2\beta}(W_{t_0+h}-W_{t_0}), \bar Y+\sqrt{2\beta}(W_{t_0+h}-W_{t_0}))\Bigr]\\
& \qquad \leq  \E\left[ \int_{\R^d} \tilde u_{t_0+h}(x) \tilde m_{t_0+h}(x)dx\right]  -  \E\left[\tilde u^\eta_{t_0+h}(\bar Y)\right]+h^2\\
&  \qquad =   \hat U^\eta(\bar X, \bar Y)+ \E\left[ \int_{\R^d} (\tilde u_{t_0+h}(x) m_{t_0+h}(x)- \tilde u_{t_0}(x) m_{t_0}(x))dx\right]\\ 
& \qquad \qquad - \E\left[\tilde u^\eta_{t_0+h}(\bar Y)-\tilde u^\eta_{t_0}(\bar Y)\right] + h^2. 
\end{split}
\ee
We analyze the  two middle terms in the right-hand side of \eqref{lerhzilrneskd1} separately. 
\vs
A standard computation gives 
\be\label{lerhzilrneskd2}
\begin{split}
& \E\left[ \int_{\R^d} (\tilde u_{t_0+h}(x) m_{t_0+h}(x)- \tilde u_{t_0}(x) m_{t_0}(x))dx\right]\\
& = \E\Big[\int_{t_0}^{t_0+h} \int_{\R^d} (\tilde u_t(x)+\tilde H_{t_0,t}(D\tilde u_t(x),x)-D_p\tilde H_{t_0,t}(D\tilde u_t(x),x)\cdot D\tilde u_t(x)\\
& \qquad -\tilde F_{t_0,t}(x,\tilde m_t)) \ \tilde m_t(x)dx dt\Big]\\
& \leq  h \int_{\R^d} (\tilde u_{t_0}(x)+ H(D\tilde u_{t_0}(x),x)-D_p H(D\tilde u_{t_0}(x),x)\cdot D\tilde u_{t_0}(x)\\
&\qquad - F(x, \bar m_0)) \   \bar m_0(x)dx +o(h),  
\end{split}
\ee
where the last inequality comes from the semiconcavity of $\tilde u$. 
We also have
\[
\begin{split}
\E\left[ \tilde u^\eta_{t_0+h}(\bar Y)-\tilde u^\eta_{t_0}(\bar Y)\right] 
= & \E\Big[\int_{t_0}^{t_0+h} \tilde u^\eta_t(\bar Y)+ \xi_\eta \ast\left( \tilde H_{t_0,t}(D\tilde u_t(\cdot),\cdot)\right)(\bar Y)\\
&-\xi_\eta\ast \tilde F_{t_0,t}(\cdot,\tilde m_t)(\bar Y)\ dt \Big]. 
\end{split}
\]
Using that $D\tilde u_t$ is bounded, $H$ is locally Lipschitz continuous and convex in the first variable, $F$ is globally Lipschitz continuous and $\xi$ has a compact support, we get 
\begin{align*}
&\E\left[ \tilde u^\eta_{t_0+h}(\bar Y)-\tilde u^\eta_{t_0}(\bar Y)\right] 
\geq  \E\left[\int_{t_0}^{t_0+h} \tilde u_t(\bar Y)+ \tilde H_{t_0,t}(D\tilde u^\eta_t(\bar Y),\bar Y)- \tilde F_{t_0,t}(\bar Y,\tilde m_t)\ dt \right]-C\eta h. 
\end{align*}
Since the map $Y\to -\E\left[ \tilde u^\eta_{t_0}(Y)\right] -\Phi(\bar X,Y)$ has a minimum at $\bar Y$, we know that  $$D_Y\Phi(\bar X,\bar Y)= -D\tilde u^\eta_{t_0}(\bar Y) \ a.s..$$
Recalling  that $D\tilde u$ is globally bounded and the change of notation at the beginning of this part, yields  \eqref{cond2TER+}. 
\vs 
Moreover the last inequality can be rewritten as 
\be\label{lerhzilrneskd3}
\begin{split}
&\E\Big[ \tilde u^\eta_{t_0+h}(\bar Y)-\tilde u^\eta_{t_0}(\bar Y)\Big] \\
&\geq  \E\Big[ \tilde u_{t_0}(\bar Y)+  H(- D_Y\Phi(\bar X,\bar Y),\bar Y)-F(\bar Y, \bar m_0)\ dt \Big] -C\eta h-o(h),
\end{split}
\ee
because, in view of the  the semiconcavity of $\tilde u$, $t\to D\tilde u^\eta_t(x)$ is continuous in $L^1_{loc}$ at $t_0$. 
\vs
Combining  \eqref{lerhzilrneskd1}, \eqref{lerhzilrneskd2} and \eqref{lerhzilrneskd3} completes the proof. 

\end{proof}

\begin{proof}[The proof of Lemma~\ref{lem.laesndfm2}.] 
Set 
$$
Z_t= (X^N_{t_0+h}+\sqrt{2\beta}(W_{t_0+h}-W_{t_0}), \bar Y+\sqrt{2\beta}(W_{t_0+h}-W_{t_0})).
$$  
The map  $t\to X^N_t$ is Lipschitz continuous in $L^2$ and solves \eqref{defXN}. Hence,  for any bounded stopping time $\tau\geq t_0$ we have 
\begin{align*}
\Phi(Z_\tau) & =\Phi(\bar X, \bar Y)+  \int_{t_0}^{\tau} (-\E\left[ D_p\tilde H^N_{t_0,t}(D\tilde u^N_t(X^N_t), X^N_t)\cdot D_X\Phi(Z_t)|W\right] \\
&\qquad + \beta \sum_{k=1}^d D^2_{(X,Y)}\Phi(Z_t)((e_k,e_k),(e_k,e_k)) \ )\ dt
+ \sqrt{2\beta} \int_{t_0}^{\tau} (D_X\Phi(Z_t)+D_Y\Phi(Z_t))dW_t. 
\end{align*}
It follows from a standard localization argument that 
\[
\begin{split}
& \E\left[ \Phi(Z_{t_0+h}) \right] =\Phi(\bar X, \bar Y)+   \int_{t_0}^{t_0+h} (-\E\left[ D_p\tilde H^N_{t_0,t}(D\tilde u^N_t(X^N_t), X^N_t)\cdot D_X\Phi(Z_t)\right] \\
&\hskip1in + \beta \Lambda (D^2_{(X,Y)}\Phi(Z_t)))dt \geq \\ 
&\hskip.5in  \Phi(\bar X, \bar Y)+ h \E\Bigl[  - D_p H (D\tilde u_{t_0}(\bar X),\bar X)\cdot D_X\Phi(\bar X,\bar Y)\Bigr]+\beta h \Lambda (D^2_{(X,Y)}\Phi(\bar X,\bar Y))+o(h).
\end{split}
\]
\end{proof}

\subsection*{Formulation for the gradient of the solution}\label{subsec.gradient}

We explain in more detail how to formulate all the results of this section in term of the derivative $D_xU$ of $U$. As pointed out several times already, this formulation is the natural one in our framework. Let us underline also that the knowledge of $D_xU$ is central in the applications since the vector field $-D_pH(D_xU(x,m),x)$ is the optimal feedback of the MFG problem. 
\vs
We begin noticing that the gradient $W=(W_1, \dots, W_d)=D_xU$ of a solution $U$ to \eqref{ME2} satisfies, at least formally and for each $i=1.\ldots,d$,  
\be\label{ME3}
\begin{split}
& W_i(x,m)-\beta \Delta W_i(x,m) + D_pH(W(x,m),x)\cdot D_xW_i(x,m)+ D_{x_i}H(W(x,m),x)\\[1.5mm] 
& + \int_{\R^d} D_{m}W_i(x,m,y) \cdot D_pH(W(y,m),y)m(dy)\\[1.5mm] 
& -\beta \Bigl( \int_{\R^d} Tr(D^2_{ym}W_i(x,m,y))m(dy) +2\int_{\R^d} Tr(D^2_{xm} W_i(x,m,y))m(dy)\\[1.5mm]  
& +\int_{\R^{2d}} Tr(D^2_{mm}W_i(x,m,y,y'))m(dy)m(dy')\Bigr) = 
F_{x_i}(x,m).
\end{split}
\ee

Mimicking Definition \ref{def.wealsol} we introduce, for $(X,Y)\in L^\infty_{ac}\times L^2$,  
$$
\hat W(X,Y)= \E\left[\int_0^1 W((1-t)X+tY, \mathcal L(X))dt\right], \;
\hat W^\ep(X,Y)= \hat W(X,Y)+ \ep  \left\|\frac{d\mathcal L(X)}{d\lambda}\right\|_{\infty}. 
$$
We define the notion of weak solution to \eqref{ME3}: 
\begin{defn}\label{def.wealsolBIS}
A  map $W:\R^d\times \mathcal P_2(\R^d)\to \R^d$ is a weak solution of the master equation \eqref{ME3}  if  (i)~$W$ is globally bounded, $m\to W(\cdot,m)$ is continuous in $L^1_{loc}(\R^d)$, $x\to W(x,m)$ is irrotational for any $m$ and satisfies, for $x,y\in \R^d$ and  $m\in \Pw$ 
$$
(W(x,m)-W(y,m))\cdot (x-y)\leq C|x-y|^2, 
$$
and (ii)~there exists a constant $C>0$ such that, for all $( X, Y)\in L^\infty_{ac}\times L^2$, any $\ep\in (0,1)$ and all 
$(\mathcal X, (p_X,p_Y))\in  \bar D^{2,-}\hat W^\ep(X,Y)$, 
\begin{align*}
0 \leq &\;  \hat W( X,  Y) -\beta \Lambda \left( \mathcal X \right)
 - \E\left[F( X, \mathcal L( X))-F( Y, \mathcal L( X))\right] \notag\\
& +\E\left[ H(  W( X, \mathcal L( X)),X)-H(-p_Y,Y)\right]\\
& - \E \left[ (W( X, \mathcal L( X))-p_X)\cdot D_pH( W( X, \mathcal L( X)),X)\right] +C\ep\left(1+ \left\|\frac{d\mathcal L(X)}{d\lambda}\right\|_{L^\infty(\R^d)}\right). \notag
\end{align*}
\end{defn}

With this definition in mind, Proposition \ref{takis3.1} can be restated as follows. 
\begin{prop}\label{takis3.1BIS} Assume that $H:\R^d\times \R^d\to \R$ is of class $C^1$ and $F:\R^d\times \mathcal P_1(\R^d)\to \R$ is continuous and class $C^1$ in the space variable. If $W:\R^d\times \Pw\to \R^d$ is a weak solution of \eqref{ME3} and  $W$, $D_xW, D_mW, D_{mm}W$ and $D_{xm}W$
are continuous in $x$ and $m$, then $W$ is a classical solution to  \eqref{ME3}. 
\end{prop} 

The proof is the same as the one of Proposition \ref{takis3.1}.  Simply  notice that all the expressions involving $U$ in the proof actually only involve $D_xU$. In the same way, we have the following reformulation of the uniqueness of the weak solution.

\begin{thm}\label{thm.uniqueBIS}  Assume that $H:\R^d\times \R^d\to \R$ is locally Lipschitz continuous and $F:\R^d\times \mathcal P_1(\R^d)\to \R$ is Lipschitz continuous and monotone. Then there exists at most one solution of  the master equation \eqref{ME3}. 
\end{thm}

The existence of a weak solution for \eqref{ME3} is a straightforward application of Theorem \ref{thm.existence}. 

\section{Relation between the two definitions for the first-order master equation} 

We revisit  the first-order master equation \eqref{ME1} and show directly that, in this case, the definition in the Hilbert space (Definition \ref{def.wealsol}) is equivalent  to the one  on the space of measures (Definition \ref{d1}). One direction is rather straightforward while the opposite is more complicated.

\begin{thm} \label{thm.equivalence} A map $U$ is a weak solution of \eqref{ME2} with $\beta =0$ in the sense of Definition \ref{def.wealsol} if and only if $U$ is a weak solution of \eqref{ME1} in the sense of Definition \ref{d1}. 
\end{thm}

We split the proof in two propositions, each one stating one implication in the equivalence claimed by the theorem. 

\begin{prop}\label{prop1} Let $U$ be a weak solution of \eqref{ME2} with $\beta =0$ in the sense of Definition \ref{def.wealsol}. Then $U$ is a weak solution of \eqref{ME1} in the sense of Definition \ref{d1}. 
\end{prop} 

\begin{proof} Let  $\phi$ be a Lipschitz continuous map, $\tilde m\in  \Pw\cap L^\infty(\R^d) $, $\hat m\in \Pw$ and assume that $m_0\in  \Pw\cap L^\infty(\R^d) $ minimizes 
$$
m \to \int_{\R^d} (U(x,m)-\phi(x))(m(x)-\tilde m(x))dx +\ep(\dw(m,\hat m)+ \|m\|_\infty).
$$
Fix $X_0\in L^2$ and $\tilde Y,\hat Y\in L^2$ be such that $\mathcal L(X_0)=m_0$, $\mathcal L(\tilde Y)=\tilde m$ and $\mathcal L(\hat Y)=\hat m$. Then, for any $\delta >0$, the map
\begin{align*}
X\to&  \E\left[ U(X,\mathcal L(X))-U(\tilde Y,\mathcal L(X))-\phi(X)+\phi(\tilde Y)+\delta |X-X_0|^2\right]+\ep\| X-\hat Y\|_2 \\
& +\ep\left\|\frac{d\mathcal L(X)}{d\lambda}\right\|_{\infty}
\end{align*}
has a unique minimum at $X_0$. 
\vs
Fix $\alpha>0$ and let $\phi_\alpha$ be a standard regularization of $\phi$, such that $D\phi_\alpha$ is uniformly bounded and converges a.e. to $D\phi$. It follows from  Stegall's Lemma  that there exist  $p_X$, $p_Y$ with $\|p_X\|_{2}+\|p_Y\|_{2}<\alpha$ and such that the map 
\begin{align*}
(X,Y) \to & \E\Bigl[ U(X,\mathcal L(X))-U( Y,\mathcal L(X))-\phi_\alpha(X)+\phi_\alpha( Y)+\delta |X-X_0|^2+\frac{1}{2\alpha} |Y-\tilde Y|^2 \\
& \qquad  - p_X \cdot X-p_Y\cdot Y\Bigr] +\ep(\alpha+\E\left[|X-\hat Y|^2\right])^{1/2} +\ep\left\|\frac{d\mathcal L(X)}{d\lambda}\right\|_{\infty}
\end{align*}
 has a minimum at $(X_\alpha,Y_\alpha)$, and  $Y_\alpha\to \tilde Y$ and $X_\alpha\to X_0$ as $\alpha\to 0$. 
 \vs
 
 It follows from  Definition \ref{def.wealsol} that 
\begin{align}\label{oauzlqensmrdc}
0 \leq &\;  \hat U( X_\alpha,  Y_\alpha) 
 - \E\left[F( X_\alpha, \mathcal L( X_\alpha))-F( Y_\alpha, \mathcal L( X_\alpha))\right] \notag\\[1.2mm]
& +\E\left[ H(  D_xU( X_\alpha, \mathcal L( X_\alpha)),X_\alpha)-H(-p_{Y_\alpha},Y_\alpha)\right]\\
& - \E \left[ (D_xU( X_\alpha, \mathcal L( X_\alpha))-p_{X_\alpha})\cdot D_pH( D_xU( X_\alpha, \mathcal L( X_\alpha)),X_\alpha)\right] +C\ep\left(1+ \left\|\frac{d\mathcal L(X_\alpha)}{d\lambda}\right\|_{\infty}\right),\notag
\end{align}
where 
\be\label{deflpah}
p_{X_\alpha}= D\phi(X_\alpha) +2\delta (X_\alpha-X_0)-\ep(\alpha+\E\left[|X_\alpha-\hat Y|^2\right])^{-1/2}(X_\alpha-\hat Y) +p_X
\ee
and
$$
p_{Y_\alpha} =- \frac{Y_\alpha-\tilde Y}{\alpha} -D\phi(Y_\alpha)+p_Y.
$$
In view of the  optimality of $Y_\alpha$,  $-p_{Y_\alpha}\in D^+_xU(Y_\alpha,\mathcal L(X_\alpha))$ a.s. and, therefore, $p_{Y_\alpha}$ is bounded in $L^\infty$ since $U$ is uniformly Lipschitz continuous in the first variable. It follows that, up to a subsequence denoted in the same way as the full family, the $p_{Y_\alpha}$'s converge,  as $\alpha\to 0$, weakly in $L^2$ to some $p_{\tilde Y}$.  Since $U$ is uniformly semiconcave in space, it follows that $-p_{\tilde Y}\in D^+_xU(\tilde Y,\mathcal L(\tilde X))$. 
Finally, given that  $\tilde Y$ has an absolutely continuous density, we infer that  $-p_{\tilde Y}= D_xU(\tilde Y,\mathcal L(\tilde X))$ a.s.. 
\vs

Using the 
convexity of $H$, we get 
$$
\E\left[H(D_xU(\tilde Y,\mathcal L(\tilde X)),\tilde Y)\right]\leq  \liminf_{\alpha \to 0} \E\left[ H(-p_{Y_\alpha},Y_\alpha)\right].
$$
The other terms in \eqref{oauzlqensmrdc} easily pass to the limit. Indeed,  recalling that the density of the law of $X_\alpha$ converges to $m_0$ in $L^\infty-$weak-$\ast$  and noticing that the term in $\ep$ in \eqref{deflpah} is uniformly bounded by $\ep$,  as expected we obtain
\begin{align*}
0 \leq &\;  \hat U( X_0,  \tilde Y) 
 - \E\left[F( X_0, \mathcal L( X_0))-F( \tilde Y, \mathcal L( X_0))\right] \notag\\
& +\E\left[ H(  D_xU( X_0, \mathcal L( X_0)),X_0)-H(D_xU(\tilde Y,\mathcal L(\tilde X),\tilde Y)\right]\\
& - \E \left[ (D_xU( X_0, \mathcal L( X_0))- D\phi(X_0)\cdot D_pH( D_xU( X_0, \mathcal L( X_0)),X_0)\right]\notag\\
&  +C\ep\left(1+ \left\|\frac{d\mathcal L(X_0)}{d\lambda}\right\|_{\infty}\right). \notag
\end{align*}
\end{proof}

We now consider the other direction. A already mentioned, the argument is much more intricate than the one for Proposition \ref{prop1}. The main difficulty is how to transform the subdifferential in the Hilbert space  in the definition into a test function in the space of measures. This question has been investigated in Gangbo and Tudorascu \cite{BaTu19}  in the setting of Hamilton-Jacobi equations (see also Alfonsi and Jourdain \cite{AlJo20} for a related topic) and we largely use these ideas although in a slightly different  context.

\begin{prop}
If $U$ is a weak solution of \eqref{ME1} in the sense of Definition \ref{d1}, then $U$ satisfies \eqref{ME2} with $\beta=0$ in the sense of Definition \ref{def.wealsol}. 
\end{prop}

\begin{proof}
Fix $\ep>0$ and a  $C^2$-test function $\Phi: L^2\times L^2\to \R$ and assume that the map 
\begin{equation}\label{moizaedf}
\begin{split}
\ds & (X,Y) \to  \E\left[ U(X, \mathcal L(X))-U(Y, \mathcal L(X))\right]
-\Phi(X,Y)
+\ep \left\|\frac{d\mathcal L(X)}{d\lambda}\right\|_{\infty}\\[1.5mm]
& \text{achieves a strict minimum $I$ at $(\bar X, \bar Y)\in L^\infty_{ac}\times L^2$.}
\end{split}
\end{equation}

The first step consists in finding a perturbation ensuring that, at the minimum $(\bar X_\delta, \bar Y_\delta)$, we have in addition that $\bar Y_\delta \in L^\infty_{ac}$. 
\vs

Fix $\delta>0$. Stegall's lemma yields  $p_X,p_Y\in L^2$ with $\|p_X\|_{2}+\|p_Y\|_{2}\leq \delta$ and $(\bar X_\delta, \bar Y_\delta)\in L^\infty_{ac}\times L^\infty_{ac}$ such that  
\begin{equation*}
\begin{split}
\ds  (X,Y) \to & \E\left[ U(X, \mathcal L(X))-U(Y, \mathcal L(X))-p_X \cdot X-p_Y  \cdot Y+|X-\bar X|^2+|Y-\bar Y|^2\right]
 \\
&-\Phi(X,Y)+\ep \left\|\frac{d\mathcal L(X)}{d\lambda}\right\|_{\infty}
+\delta \left\|\frac{d\mathcal L(Y)}{d\lambda}\right\|_{\infty}\\[1.5mm]
& \text{achieves a minimum $I_\delta$ at $(\bar X_\delta, \bar Y_\delta)\in L^\infty_{ac}\times L^\infty_{ac}$.}
\end{split}
\end{equation*}
Note that, as $\delta \to 0,$  $I_\delta\to I$ and, hence,   $(\bar X_\delta,\bar Y_\delta) \to (\bar X,\bar Y)$  in $L^2$ and   
\be\label{ulaenjsdfmk}
\lim_{\delta \to 0}  \left\|\frac{d\mathcal L(\bar X_\delta)}{d\lambda}\right\|_{\infty}=
 \left\|\frac{d\mathcal L(\bar X)}{d\lambda}\right\|_{\infty} \ \text{and} \  \lim_{\delta\to 0} \delta \left\|\frac{d\mathcal L(\bar Y_\delta)}{d\lambda}\right\|_{\infty}=0 .
 \ee
 \vs
We also note  that the $D_xU(\bar Y_\delta,\mathcal L(\bar X_\delta))$'s are  bounded in $L^\infty$ and, therefore, converge (up to a sequence that we denote in the same way) in $L^\infty-$weak $\ast$ to a random variable $Z \in \sigma(\bar Y)$. The measurability of $Z$ is a consequence of the facts that $D_xU(\bar Y_\delta,\mathcal L(\bar X_\delta)) \in\sigma(\bar Y_\delta)$ and $\bar Y_\delta \to \bar Y$.
\vs
%
 We claim that 
 \be\label{takis421}Z= -D_Y\Phi(\bar X, \bar Y).\ee
 Indeed,  fix $\phi\in C^1_c(\R^d;\R^d)$. In view of the  optimality of $\bar Y_\delta$, for $h>0$ we have 
\begin{align}\label{mezjsdc}
 &\E\left[-U(\bar Y_\delta+h\phi(\bar Y_\delta), \mathcal L(\bar X_\delta))-p_Y (\bar Y_\delta+h\phi(\bar Y_\delta))+|(\bar Y_\delta+h\phi(\bar Y_\delta))-\bar Y|^2\right]\notag\\
 & \qquad -\Phi(\bar X_\delta,(\bar Y_\delta+h\phi(\bar Y_\delta))) +\delta \left\|\frac{d\mathcal L((\bar Y_\delta+h\phi(\bar Y_\delta)))}{d\lambda}\right\|_{\infty}\\
 &\geq \E\left[-U(\bar Y_\delta, \mathcal L(\bar X_\delta))-p_Y \bar Y_\delta+|\bar Y_\delta-\bar Y|^2\right]-\Phi(\bar X_\delta,\bar Y_\delta) +\delta \left\|\frac{d\mathcal L(\bar Y_\delta)}{d\lambda}\right\|_{\infty}. \notag
\end{align}
Recalling that the density of the law of $(\bar Y_\delta+h\phi(\bar Y_\delta))$ is given by $|\text{det}(J(Id+h\phi)^{-1})|m\circ (Id+h\phi)^{-1}$, $m$ being  the law of $\bar Y_\delta$ and $J$ the Jacobian matrix, we have that 
$$
\left\|\frac{d\mathcal L((\bar Y_\delta+h\phi(\bar Y_\delta)))}{d\lambda}\right\|_{\infty} \leq (1+Ch\|D\phi\|_\infty)\|\left\|\frac{d\mathcal L(\bar Y_\delta)}{d\lambda}\right\|_{\infty}.
$$
Hence, dividing \eqref{mezjsdc} by $h$ and letting $h\to0$ we get 
\begin{align*}
&\E\left[(-D_xU(\bar Y_\delta, \mathcal L(\bar X_\delta)) -p_Y\cdot +2(\bar Y_\delta-\bar Y)-D_Y\Phi(\bar X_\delta,\bar Y_\delta))\cdot \phi(\bar Y_\delta)\right]\\
&\qquad \geq -C\delta \|\phi\|_{C^1} \left\|\frac{d\mathcal L(\bar Y_\delta)}{d\lambda}\right\|_{\infty}. 
\end{align*}
\vs

Next,  we let $\delta \to 0$. Since $D_xU(\bar Y_\delta,\mathcal L(\bar X_\delta)) \to Z$  in  $L^\infty-$weak $\ast$, $(\bar X_\delta,\bar Y_\delta)\to (\bar X,\bar Y)$  in $L^2$, 
and \eqref{ulaenjsdfmk} holds, we find 
\begin{align*}
\E\left[-Z\cdot \phi(\bar Y)-D_Y\Phi(\bar X,\bar Y)\cdot \phi(\bar Y)\right] \geq 0. 
\end{align*}

In Lemma~\ref{lem.barYmeas} below we prove that $D_Y\Phi(\bar X, \bar Y)$ is $\sigma(\bar Y)$ measurable. Hence, \eqref{takis421} holds.
\vs

We note for later use that, in view of the  convexity of $H$ with respect to the first variable,  we find  that

\be\label{maejznrsdfgc}
\E\left[ H(-D_Y\Phi(\bar X, \bar Y), \bar Y)\right] \leq \liminf_{\delta\to 0} \E\left[ H(D_xU(\bar Y_\delta, \mathcal L(\bar X)), \bar Y_\delta)\right].
\ee
\vs

We now start the second part of the proof, in which $\bar Y_\delta$ is fixed and where we use the fact that $U$ is a solution of \eqref{ME1}. We set $m_0= \mathcal L(\bar X_\delta)$ and $\tilde m= \mathcal L(\bar Y_\delta)$. Then $\bar X_\delta$ is a minimum point of 
$$
 X \to  \E\left[ U(X, \mathcal L(X))-U(\bar Y_\delta, \mathcal L(X))\right]
-\Phi(X,\bar Y_\delta)
+\ep  \left\|\frac{d\mathcal L(X)}{d\lambda}\right\|_{\infty}. 
$$
For $\mu\in \Pw$ and  $X\in L^2$ let 
$$
W(\mu)=  \int_{\R^d} U(x, \mu)(\mu-\tilde m)+\ep\left\|\frac{d\mu}{d\lambda}\right\|_{\infty} \ \text{and} \ \tilde W(X)=W(\mathcal L(X)).
$$

Since  $\tilde W-\Phi(\cdot, \bar Y_\delta)$ has a minimum at $\bar X_\delta$ and $\Phi\in C^2$  it follows that, for some constant $C>0$ and all $X\in L^2$,  
\be\label{uhlzqnesdf}
\tilde W(X) \geq \tilde W(\bar X_\delta)+\E\left[ D_X\Phi(\bar X_\delta,\bar Y_\delta)\cdot (X-\bar X_\delta)\right] -C\|X-\bar X_\delta\|_2^2. 
\ee
The main difficulty is to replace $D_X\Phi(\bar X_\delta,\bar Y_\delta)$ by a map of the form $D\phi(\bar X_\delta)$ for some $\phi \in C^1(\R^d)$. 
\vs

It turns out that, although we are not able to find such $\phi$,  we have the following result, which is largely borrowed from ideas of \cite{BaTu19} and \cite{AlJo20} (recall that $m_0=\mathcal L(\bar X_\delta)$). Its proof is presented after the end of the ongoing one. 

\begin{lem}\label{takis513} There exists a map $h\in L^2_{m_0}(\R^d)$ and a sequence $\phi_n\in C^\infty_c(\R^d)$ such that, as $n\to \infty$,  $D\phi_n \to h$ in $L^2_{m_0}(\R^d, \R^d)$ and, for all $v\in L^2_{m_0}(\R^d, \R^d)$, 
\be\label{qsndlrthmf}
\E\left[ v(\bar X_\delta)\cdot D_X\Phi(\bar X_\delta,\bar Y_\delta)\right] = \E\left[ v(\bar X_\delta)\cdot h(\bar X_\delta)\right].
\ee
\end{lem}

We also need the following fact. Its proof is given later in this section.

\begin{lem}\label{takis514} If $n$ is sufficiently large, the map
$$
\Pw \ni m\to \int_{\R^d} (U(x,m)-\phi_n(x))(m(dx)-\tilde m(dx)) +\ep  \Big({\bf d}_2(m,m_0)+  \left\|m\right\|_{\infty}\Big)
$$
has a local minimum at $m_0 \in \Pw\cap L^\infty(\R^d)$.
\end{lem}

Since  $U$ is a solution of  \eqref{ME1} in the sense of Definition \ref{d1}, in view of the previous lemmata, we find  \begin{align*}
& \ds \int_{\R^d} U(x, m) (m_0(x)-\tilde m(x))dx + \int_{\R^d} H(D_xU(x, m_0),x)(m_0(x)-\tilde m(x))dx \\
& \qquad \ds - \int_{\R^d} (D_xU(y, m_0)-D\phi_n(y))\cdot D_pH(D_xU(y,m),y )m_0(dy)\\
& \qquad\qquad \ds \geq 
\int_{\R^d} F(x,m_0)(m_0(x)-\tilde m(x))dx- C\ep(1+\|m_0\|_{L^\infty(\R^d)}^2). 
\end{align*}
Moreover, since  $D\phi_n \to h$ in $L^2_{m_0}(\R^d)$, letting  $n\to \infty$ yields 
\begin{align*}
& \ds \int_{\R^d} U(x, m) (m_0(x)-\tilde m(x))dx + \int_{\R^d} H(D_xU(x, m_0),x)(m_0(x)-\tilde m(x))dx \\
& \qquad \ds - \int_{\R^d} (D_xU(y, m_0)-h(y))\cdot D_pH(D_xU(y,m),y )m_0(dy)\\
& \qquad\qquad \ds \geq 
\int_{\R^d} F(x,m_0)(m_0(x)-\tilde m(x))dx- C\ep(1+\|m_0\|_{\infty}^2).
\end{align*}
Note that in view of  \eqref{qsndlrthmf} and of the definition of $m_0$ and $\tilde m$ the above can be rewritten as 
\begin{align*}
& \ds \E\left[ U(\bar X_\delta, \mathcal L(\bar X_\delta)) - U(\bar Y_\delta,  \mathcal L(\bar X_\delta))  + H(D_xU(\bar X_\delta, \mathcal L(\bar X_\delta)), \bar X_\delta)- H(D_xU(\bar Y_\delta, \mathcal L(\bar X_\delta)), \bar Y_\delta) \right] \\
& \qquad \ds - \E\left[ (D_xU(\bar X_\delta, \mathcal L(\bar X_\delta))-D_X\Phi(\bar X_\delta, \bar Y_\delta))\cdot D_pH(D_xU(\bar X_\delta,\mathcal L(\bar X_\delta)),\bar X_\delta )\right]\\
& \qquad\qquad \ds \geq 
\left[ F(\bar X_\delta,\mathcal L(\bar X_\delta))- F(\bar Y_\delta,\mathcal L(\bar X_\delta))\right]- C\ep\Big(1+ \left\|\frac{d\mathcal L(\bar X_\delta)}{d\lambda}\right\|_{\infty}\Big).
\end{align*}
Recalling that $(\bar X_\delta,\bar Y_\delta) \to (\bar X,\bar Y)$  in $L^2$ and  that \eqref{ulaenjsdfmk} and \eqref{maejznrsdfgc} hold, we  obtain easily \eqref{cond2BIS} by letting $\delta\to 0$. 

\end{proof}

In the proof above we used the following fact. 

\begin{lem}\label{lem.barYmeas} Let $\bar Y$ be defined as in \eqref{moizaedf}. Then $D_Y\Phi(\bar X, \bar Y)$ is $\sigma(\bar Y)-$measurable. 
\end{lem}

\begin{proof} Fix a standard mollifier  $\rho$ and, for $\delta >0$, set  $\rho_\delta(x)=\delta^{-d}\rho(x/\delta)$. It follows that there exists $p_Y$ with $\|p_Y\|_2<\delta$ such that 
$$
Y\to   \E\left[-\rho_\delta \ast U(\cdot, \mathcal L(\bar X))(Y)-p_Y\cdot Y+|Y-\bar Y|^2\right] -\Phi(\bar X,Y)
$$
has a minimum at some $Y_\delta$. 
\vs
Since $\bar Y$ is a minimum of 
$$
Y\to   \E\left[-U(Y, \mathcal L(\bar X))\right] -\Phi(\bar X,Y)
$$
and  $U$ is uniformly Lipschitz continuous in the first variable, it follows that, as $\delta \to 0$,  $Y_\delta \to  \bar Y$ in $L^2$. 
\vs 
The optimality condition for $Y_\delta$ reads 
$$
-\rho_\delta \ast D_xU(\cdot, \mathcal L(\bar X))(Y_\delta)-p_Y+2(Y_\delta-\bar Y) -D_Y\Phi(\bar X,Y_\delta)=0, 
$$
so that $D_Y\Phi(\bar X,Y_\delta)+p_Y$ is measurable with respect to $\sigma(Y_\delta, \bar Y)$. 
\vs

Letting  $\delta\to 0$ yields the claim.

\end{proof}

We conclude with the proofs of the two lemmata used above.

\begin{proof}[The proof of Lemma~\ref{takis513}] Let  $p_X= D_X\Phi(\bar X_\delta,\bar Y_\delta)$ and $\mu=\mathcal L((\bar X_\delta,p_X))$, and denote by 
$\nu_x$ the conditional law of $p_X$ given $\bar X_\delta=x$. Then $$\mu(dx,dy)= m_0(dx) \nu_x(dy).$$ 
  
 Let $Q_1=[0,1]^d$ and $\lambda$ be the Lebesgue measure on $Q_1$. For $m_0-$a.e. $x\in \R^d$, there exists a unique gradient $\psi_x:\R^d\to \R^d$ of a convex function such that $\nu_x=\psi_x\sharp \lambda$. It then follows from the  continuity of the optimal transport map with respect to the target measure that the map $(x,y)\to \psi_x(y)$ is measurable. 
 \vs
 
Consider a random variable $Z$ with uniform law on $Q_1$ which is independent of $\bar X_\delta$. It follows that 
 the law of $(\bar X_\delta, \psi_{\bar X_\delta}(Z))$ is equal to $\mu$. 
 \vs
 Indeed, for any $f\in C^0_b(\R^d\times \R^d)$, we have 
\begin{align*}
\E\left[ f(\bar X_\delta, \psi_{\bar X_\delta}(Z))\right]&  = \int_{\R^{d}\times Q_1} f(x,\psi_x(z)) m_0(dx)dz= \int_{\R^d\times \R^d} f(x,y) \psi_x\sharp \lambda (dy) m_0(dx)\\
&= 
\int_{\R^d\times \R^d} f(x,y)\mu(dx,dy).
\end{align*}
In particular, since $(\bar X_\delta, \psi_{\bar X_\delta}(Z))$  and $(\bar X_\delta,p_X)$ have the same law, for any measurable and bounded vector field $v:\R^d\to \R^d$, we have 
$$
\E\left[ v(\bar X_\delta) \cdot p_X\right] = \E\left[ v(\bar X_\delta)\cdot \psi_{\bar X_\delta}(Z)\right] = \E\left[ v(\bar X_\delta)\cdot  h(\bar X_\delta)\right], 
$$
where $h(\bar X_\delta)$ is the conditional expectation of $\psi_{\bar X_\delta}(Z)$ given $\bar X_\delta$, which, in view of the independence of $\bar X_\delta$ and $Z$, is equal to
$$
h(\bar X_\delta)= \E\left[ \psi_x(Z)\right]_{x=\bar X_\delta}= \int_{Q_1} \psi_{\bar X_\delta}(z)dz.
$$

The aim  is to prove that the measurable map $h:\R^d\to \R^d$ actually belongs to $\mathcal{T}^2_{m_0}\Pw$, which is the closure in $L^2_{m_0}(\R^d,\R^d)$ of the set $\{ D\phi,\; \phi\in C^\infty_c(\R^d)\}$. 
\vs

For this we recall that the orthogonal complement of $\mathcal{T}^2_{m_0}\Pw$  in $L^2_{m_0}(\R^d,\R^d)$ is the set of vector fields $b\in  L^2_{m_0}(\R^d,\R^d)$ such that ${\rm div}(bm_0)=0$ in the sense of distributions. 
\vs

Fix $b$ as above. We claim that  $$\int_{\R^d} h(x)\cdot b(x) m_0(dx)=0.$$ 

Indeed, let  $T>0$ and note that $m_0$ is a constant-in-time solution of the continuity equation 
$$
\partial_t m +{\rm div}(mb)=0\; {\rm on}\; \R^d\times (0,T], \qquad m(0)=m_0.
$$
It follows from the classical Ambrosio's superposition principle, that there exists a Borel probability measure $\eta$ on $\Gamma=C^0([0,T], \R^d)$ such that $m_0=e_t\sharp \eta$ for any $t\in [0,T]$, $e_t$ being the evaluation map at time $t$, and,  $\eta-$a.e. $\gamma \in \Gamma$ is an absolutely continuous  solution of $\dot\gamma(t)= b(\gamma(t))$.  
\vs

Choose $t_0\in [0,T)$ such that, for $\eta-$a.e. $\gamma$, $\dot \gamma(t_0)$ exists and equals $b(\gamma(t_0))$ and 
disintegrate $\eta$ with respect to $m_0$ so that $\eta(d\gamma)= \int_{\R^d}\eta_x(d\gamma)m_0(dx)$, where, for $m_0-$a.e. $x\in \R^d$ and $\eta_x-$a.e. $\gamma\in \Gamma$, $\gamma(t_0)= x$. Note $m_{x}(t)$ the probability measure $m_{x}(t)= e_t\sharp \eta_x$. 
\vs

Fix $t\in (t_0,T]$. Arguing as above, we can find $\xi_{x,t}:Q_1\to \R^d$, which is the  gradient of a convex function, such that, for $m_0-$a.e. $x\in \R^d$, $\xi_{x,t}\sharp \lambda= m_{x}(t)$. In addition, the map $(x,t,z)\to \xi_{x,t}(z)$ is Borel measurable. 
\vs
Let $Z'$ be a random variable with uniform law on $Q_1$  such that $\bar X_\delta$, $Z$ and $Z'$ are independent, and 
apply \eqref{uhlzqnesdf} with $X= \xi_{\bar X_\delta, t}(Z')$ to get,  in view of the fact that $\mathcal L(X)=m_0$,
\begin{align*}
0 \geq & \E\left[ p_X \cdot (X-\bar X_\delta)\right] -C\|X-\bar X_\delta\|_2^2 = \E\left[ \psi_{\bar X_\delta}(Z)  \cdot (\xi_{\bar X_\delta, t}(Z')-\bar X_\delta)\right] -C\|\xi_{\bar X_\delta, t}(Z')-\bar X_\delta\|_2^2 \\
& = \E\left[ h(\bar X_\delta)  \cdot (\xi_{\bar X_\delta, t}(Z')-\bar X_\delta)\right] -C(t-t_0)^2.
\end{align*}
Note that 
\begin{align*}
& \E\left[ h(\bar X_\delta)  \cdot (\xi_{\bar X_\delta, t}(Z')-\bar X_\delta)\right]  = \int_{\R^d\times Q_1} h(x) \cdot (\xi_{x, t}(z)-x) m_0(x)dxdz \\
& \qquad =  \int_{\R^d\times \R^d} h(x) \cdot (y-x) m_0(x) \xi_{x, t}\sharp \lambda (dy) dx  = \int_{\Gamma} h(\gamma(t_0))\cdot(\gamma(t)-\gamma(0)) \eta(d\gamma)\\
 &\qquad   = (t-t_0) \int_{\R^d} h(x)\cdot b(x) m_0(x)dx +o(t-t_0). 
 \end{align*}
 Inserting the last  equality in the previous inequality we find that, for any $b\in (\mathcal{T}^2_{m_0}\Pw)^\perp$,   $$ \int_{\R^d} h(x)\cdot b(x) m_0(x)dx= 0.$$  
 
It follows that $h\in \mathcal{T}^2_{m_0}\Pw$, and  this implies that the existence of a sequence of maps $\phi_n\in C^\infty_c(\R^d)$ such that $D\phi_n \to h$ in $L^2_{m_0}(\R^d)$. 
 
\end{proof}

\begin{proof}[The proof of Lemma~\ref{takis514}] Fix $m\in \Pw$. Since  $m_0$ is absolutely continuous with respect to the Lebesgue measure, there exists a unique $\xi$, which is the gradient of a convex function, such that $\xi\sharp m_0= m$, and, in view of \eqref{qsndlrthmf}, \eqref{uhlzqnesdf} can be written, for $X= \xi(\bar X_\delta)$,  as
$$
W(m) \geq W(m_0)+\int_{\R^d} h(x) \cdot (\xi(x)-x)m_0(dx)  -C{\bf d}_2^2(m_0,m).
$$
Replacing $h$ by $D\phi_n$ we get 
$$
W(m) \geq W(m_0)+\int_{\R^d} D\phi_n(x) \cdot (\xi(x)-x)m_0(dx)  -\|D\phi_n-h\|_{L^2_{m_0}}{\bf d}_2(m_0,m) - C{\bf d}_2^2(m_0,m).
$$
Note that 

\[
\begin{split}
& \left| \int_{\R^d} \phi_n(x) (m-m_0)(dx) - \int_{\R^d} D\phi_n(x) \cdot (\xi(x)-x)m_0(dx)\right|\\
& \qquad= \left| \int_0^1 \int_{\R^d} (D\phi_n((1-t)\xi(x)+tx)-D\phi_n(x))\cdot  (\xi(x)-x)m_0(dx)\right| \leq \|D^2\phi_n\|_\infty{\bf d}^2_2(m,m_0). 
\end{split}
\]
Hence, there exist $\delta_n\to 0$, such that,  for any $m\in \Pw$,  
$$
W(m) \geq W(m_0)+\int_{\R^d} \phi_n(x) (m-m_0)(dx)  -\delta_n{\bf d}_2(m_0,m) - (\delta_n^{-1}+C){\bf d}_2^2(m_0,m).
$$

Choosing $r_n= \delta_n(\delta_n^{-1}+C)^{-1}$ yields that $m_0$ is a minimum in $B_{r_n}(m_0)$ of the map 
$$
\Pw\ni m\to W(m) -\int_{\R^d} \phi_n(x) (m-m_0)(dx) +2\delta_n {\bf d}_2(m_0,m).
$$
The definition of $W$ yields that,  if $n$ so large that $2\delta_n\leq \ep$, $m_0$ is a minimum in $B_{r_n}(m_0)$ of the map 
$$
\Pw\ni m\to \int_{\R^d} (U(x,m)-\phi_n(x))(m(dx)-\tilde m(dx)) +\ep  \Big({\bf d}_2(m,m_0)+ \left\|m\right\|_{\infty}\Big).
$$
\end{proof}

\bibliographystyle{siam}

\begin{thebibliography}{10}
\bibitem{AlJo20} Alfonsi, A., \& Jourdain, B. (2020). Squared quadratic Wasserstein distance: optimal couplings and Lions differentiability. ESAIM: Probability and Statistics, 24, 703-717.

\bibitem{AmMe21} Ambrose, D. M., \& M\'esz\'{a}ros A. R. (2021). Well-posedness of mean field games master equations involving non-separable local Hamiltonians. arXiv preprint arXiv:2105.03926.

\bibitem{bayraktar2019finite}
Bayraktar, E., Cecchin, A., Cohen, A. and Delarue, F. (2021), Finite state
  mean field games with Wright-Fisher common noise, Journal de Mathématiques Pures et Appliquées, 147, 98-162.

\bibitem{BeFrYa13} Bensoussan, A., Frehse, J., and Yam, P. (2013). {\sc Mean field games and mean field type control theory} (Vol. 101). New York: Springer.


\bibitem{bertucci2019some}
Bertucci, C., Lasry, J. M., and Lions, P.-L. (2019), Some remarks on mean
  field games, Communications in Partial Differential Equations, 44,
  pp.~205--227.
  
\bibitem{BeLaLi20} Bertucci, C., Lasry, J. M., and Lions, P. L. (2021). Master equation for the finite state space planning problem. Archive for Rational Mechanics and Analysis, 1-16.

\bibitem{Be20} Bertucci, C. (2021) Monotone solutions for mean field games master equations: finite state space and optimal stopping. Journal de l'École polytechnique-Mathématiques, 8, 1099-1132.

\bibitem{Be21} Bertucci, C. (2021). Monotone solutions for mean field games master equations: continuous state space and common noise. arXiv preprint arXiv:2107.09531.

\bibitem{Bessi20} Bessi, U. (2016). Existence of solutions of the master equation in the smooth case. SIAM Journal on Mathematical Analysis, 48(1), 204-228.


\bibitem{BLPR17} Buckdahn, R., Li, J., Peng, S., \& Rainer, C. (2017). Mean-field stochastic differential equations and associated PDEs. The Annals of Probability, 45(2), 824-878.

\bibitem{cardaliaguet2020splitting}
Cardaliaguet, P., Cirant, M., and Porretta, A. (2020), Splitting methods
  and short time existence for the master equations in mean field games, To appear in JEMS.


\bibitem{CDLL} Cardaliaguet, P., Delarue, F., Lasry, J. M., and Lions, P.-L. (2019). {\sc The Master Equation and the Convergence Problem in Mean Field Games} (AMS-201) (Vol. 381). Princeton University Press.

\bibitem{CaHa} P. Cardaliaguet and S. Hadikhanloo, S. Learning in mean field games: the fictitious play. ESAIM: Control, Optimization and Calculus of Variations, 23(2), 569--591, 2017

\bibitem{CaSo20} P. Cardaliaguet and P. E. Souganidis, P. E. (2020). On first-order mean field game systems with a common noise. To appear in Annals of Applied Probability.

\bibitem{CCD14} Chassagneux, J. F., Crisan, D., \& Delarue, F. (2014). A probabilistic approach to classical solutions of the master equation for large population equilibria. arXiv preprint arXiv:1411.3009.


\bibitem{CaDe14} Carmona, R., and Delarue, F. (2014). The master equation for large population equilibriums. In Stochastic analysis and applications 2014 (pp. 77-128). Springer, Cham.

\bibitem{CaDeBook} Carmona, R., and Delarue, F. (2018). {\sc Probabilistic Theory of Mean Field Games with Applications I-II.} Springer Nature.

\bibitem{CILuserguide} Crandall, M. G., Ishii, H., \& Lions, P. L. (1992). User's guide to viscosity solutions of second order partial differential equations. Bulletin of the American mathematical society, 27(1), 1-67.

\bibitem{CrLi1} Crandall, M. G., \& Lions, P.-L. (1985). Hamilton-Jacobi equations in infinite dimensions I. Uniqueness of viscosity solutions. Journal of Functional Analysis, 62(3), 379-396.

\bibitem{CrLi2} Crandall, M. G., \& Lions, P.-L. (1986). Hamilton-Jacobi equations in infinite dimensions. II. Existence of viscosity solutions. Journal of Functional Analysis, 65(3), 368-405.

\bibitem{Ek74} Ekeland, I.  (1974) On the variational principle. J. Math. Anal. Appl. 47, 324-353. 


\bibitem{EkLe76} Ekeland, I., \& Lebourg, G. (1976). Generic Fr\'{e}chet-differentiability and perturbed optimization problems in Banach spaces. Transactions of the American Mathematical Society, 224(2), 193-216

\bibitem{FF} M. Fabian and C. Finet. On Stegall's smooth variational principle.
Nonlinear Analysis: Theory, Methods and Applications, 66(3), 565--570, 2007.

\bibitem{FGS17} Fabbri, G., Gozzi, F., \& \'{S}wi\c{e}ch, A. (2017). {\sc Stochastic optimal control in infinite dimension.} Probability and Stochastic Modelling. Springer.


\bibitem{GaSw15} Gangbo, W., \& \'{S}wi\c{e}ch, A. (2015). Existence of a solution to an equation arising from the theory of mean field games. Journal of Differential Equations, 259(11), 6573-6643.

\bibitem{BaTu19} Gangbo, W., \& Tudorascu, A. (2019). On differentiability in the Wasserstein space and well-posedness for Hamilton-Jacobi equations. Journal de Math\'{e}matiques Pures et Appliqu\'{e}es, 125, 119-174.

\bibitem{GaMe20} Gangbo, W., \& M\'esz\'{a}ros, A. R. (2020). Global well-posedness of Master equations for deterministic displacement convex potential mean field games. arXiv preprint arXiv:2004.01660.

\bibitem{GMMZ21} Gangbo, W., M\'esz\'{a}ros, A. R., Mou, C., \& Zhang, J. (2021). Mean Field Games Master Equations with Non-separable Hamiltonians and Displacement Monotonicity. arXiv preprint arXiv:2101.12362.

\bibitem{huang2006large} Huang, M., Malham\'{e}, R. P., and Caines, P. E. (2006). Large population stochastic dynamic games: closed-loop McKean-Vlasov systems and the Nash certainty equivalence principle. Communications in Information \& Systems, 6(3), 221-252.


\bibitem{LLJapan} Lasry, J. M., and Lions, P.-L. (2007). Mean field games. Japanese journal of mathematics, 2(1), 229-260.


\bibitem{Lions89} Lions, P.-L. (1989). Viscosity solutions of fully nonlinear second-order equations and optimal stochastic control in infinite dimensions. III. Uniqueness of viscosity solutions for general second-order equations. Journal of Functional Analysis, 86(1), 1-18.

\bibitem{LiCoursCollege} Lions, P.-L. Courses at the Coll\`{e}ge de France. 

\bibitem{Ma20} Mayorga, S. (2020). Short time solution to the master equation of a first-order mean field game. Journal of Differential Equations, 268(10), 6251-6318.

\bibitem{MoZh} Mou, C., \& Zhang, J. (2020). Wellposedness of second-order master equations for mean field games with nonsmooth data. arXiv preprint arXiv:1903.09907.

\bibitem {St1} C. Stegall. Optimization of functions on certain subsets of Banach spaces.
Mathematische Annalen, 236(2), 171--176, 1978.

\bibitem {St2} C. Stegall. Optimization and differentiation in Banach spaces. Linear
Algebra and Its Applications, 84, 191--211, 1986.

\bibitem{Ta82} Tataru, D. (1992). Viscosity solutions of Hamilton-Jacobi equations with unbounded nonlinear terms. Journal of Mathematical Analysis and Applications, 163(2), 345-392.

\end{thebibliography}

%
%
%
%
\end{document}